\documentclass[article]{amsart}
\usepackage{color}
\usepackage{cases}
\usepackage{amsmath, amsfonts, pifont, amssymb}
\usepackage{enumerate}
\usepackage{verbatim}
\usepackage[bookmarks=true]{hyperref}
\usepackage{geometry}
\newtheorem{theorem}{Theorem}[section]
\newtheorem{lemma}[theorem]{Lemma}
\newtheorem{proposition}[theorem]{Proposition}
\newtheorem{corollary}[theorem]{Corollary}

\newcommand{\f}{\frac}
\newcommand{\beq}{\begin{equation}}
\newcommand{\eeq}{\end{equation}}
\newcommand{\beqq}{\begin{equation*}}
\newcommand{\eeqq}{\end{equation*}}

\theoremstyle{definition}
\newtheorem{definition}[theorem]{Definition}

\theoremstyle{remark}
\newtheorem{remark}[theorem]{Remark}

\numberwithin{equation}{section}

\def\ep{\epsilon}

\numberwithin{equation}{section}

\begin{document}

\address{Kailong Yang
\newline \indent Chongqing National Center for Applied Mathematics, Chongqing Normal University
\newline \indent  Chongqing, China\indent }
\email{ykailong@mail.ustc.edu.cn}

\address{Zehua Zhao
\newline \indent  Department of Mathematics, University of Maryland, U.S.A.
\newline \indent Department of Mathematics and Statistics, Beijing Institute of Technology
\newline \indent MIIT Key Laboratory of Mathematical Theory and Computation in Information Security.}
\email{zzh@bit.edu.cn}

\title[On scattering asymptotics for the 2D cubic resonant system]{On scattering asymptotics for the 2D cubic resonant system}
\author{Kailong Yang and Zehua Zhao}
\maketitle

\setcounter{tocdepth}{1}
\tableofcontents

\begin{abstract}
In this paper, we prove scattering asymptotics for the 2D (discrete dimension) cubic resonant system. This scattering result was used in Zhao \cite{Z1} as an assumption to obtain the scattering for cubic NLS on $\mathbb{R}^2\times \mathbb{T}^2$ in $H^1$ space. Moreover, the 1D analogue is proved in Yang-Zhao \cite{YZ}. Though the scheme is also tightly based on Dodson \cite{D}, the 2D case is more complicated which causes some new difficulties. One obstacle is the failure of `$l^2$-estimate' for the cubic resonances in 2D (we also discuss it in this paper, which may have its own interests). To fix this problem, we establish weaker estimates and exploit the symmetries of the resonant system to modify the proof of \cite{YZ}. At last, we make a few remarks on the research line of `long time dynamics for NLS on waveguides'.
\end{abstract}
\bigskip

\bigskip


\section{Introduction}
\subsection{Statements of the main results}
In this paper, we study the initial value problem of the infinite dimensional vector-valued  resonant nonlinear Schr\"{o}dinger system
\begin{equation}\label{maineq}
\begin{cases}
i\partial_t u_j + \Delta u_j =\sum\limits_{\mathcal{R}(j)}u_{j_1}\bar{u}_{j_2}u_{j_3},\\
u_j(0) = u_{0,j},
\end{cases}
\end{equation}
with unknown $\vec{u}=\{u_j\}_{j\in\mathbb{Z}^2}$, where $u_j: \mathbb{R}^2_x \times \mathbb{R}_t \rightarrow \mathbb{C}$ is complex valued function and
$$\mathcal{R}(j)=\{(j_1,j_2,j_3)\in (\mathbb{Z}^2)^3,\ \ j_1-j_2+j_3 = j,\ \ |j_1|^2-|j_2|^2 +|j_3|^2 = |j|^2\},$$
where $\mathcal{R}(j)$ is known to be the resonant relation. This model is also known as `cubic resonant Schr\"odinger system'. We \emph{emphasize} that in this paper, we regard it as `\textbf{2D} cubic resonant system' because the discrete dimension of the model \eqref{maineq} is $2$ (noticing $j \in \mathbb{Z}^2$ in \eqref{maineq}). As a comparison, the 1D analogue is studied in \cite{YZ}. More precisely, global well-posedness and scattering are obtained in \cite{YZ}, which are used in \cite{CHENG} as a crucial element. At last, we refer to \cite{CZZ,HP} for the quintic analogues of the resonant Schr\"odinger system.

We intend to study global well-posedness and scattering for \eqref{maineq} with large data, which is used in \cite{Z1} as an important assumption. The main theorem of this paper is as follows,
\begin{theorem}\label{mainthm}
The initial value problem \eqref{maineq} is globally well-posed and scattering holds for any $\vec{u}_0=\{u_{0,j}\}_{j\in \mathbb{Z}^2} \in L^2_xh^1 (\mathbb{R}^2\times \mathbb{Z}^2)$. Here we say the solution $\vec{u}(t,x)$ to \eqref{maineq} scattering means that there exist  $\vec{u}^{\pm}\in L_x^2h^1$  such that
\begin{equation}
  \lim_{t\rightarrow \pm\infty} \|e^{it\triangle}\vec{u}^{\pm}-\vec{u}(t)\|_{L_x^2h^1}=0.
\end{equation}
\end{theorem}
\begin{remark}
The cubic resonant system \eqref{maineq} can be regarded as the system analogue of cubic mass critical NLS. If one considers the special case for \eqref{maineq}, when $u_j=0$ for $j\neq 0$, the system is reduced to the 2D cubic Schr\"odinger equation
\begin{equation}
    (i\partial_t+\Delta_{\mathbb{R}^2})u=u|u|^2,
\end{equation}
which is mass-critical (in the sense: the mass of the solution is invariant under the scaling symmetry). The scattering for this model holds (see Dodson \cite{D}) thus heuristically the scattering result for \eqref{maineq}, i.e. Theorem \eqref{mainthm} is also expected to hold.
\end{remark}

\begin{remark}
The local well-posedness and small data scattering for \eqref{maineq} has been included in the appendix of \cite{Z1}.
\end{remark}

\begin{remark}
One may compare this result with \cite{YZ}, which is for the 1D analogue. We will discuss the differences explicitly later on.
\end{remark}
\subsection{Background and motivations}

Basically there are three motivations that drive us to study the cubic resonant system \eqref{maineq}.
\begin{enumerate}[\bf (1)]
\item
{\bf Weak turbulence problems.}
It is tightly related to \emph{Weak turbulence problems} (growth of Sobolev norms for NLS) where 2D cubic resonant system appears a lot especially when one studies cubic NLS on two dimensional torus. We note again that the ‘2D’ means the discrete dimension. (It can roughly viewed from the fact that Fourier transformation takes torus to lattice.) We refer to \cite{C1,Faou,Hani1} and the reference therein for more information and background.
\item
{\bf NLS problems on waveguide manifolds.}   This model is also tightly related to \emph{NLS problems on waveguide manifolds} (i.e. product spaces of this form: $\mathbb{R}^n\times \mathbb{T}^m$, which are also known semiperiodic spaces). In particular, the scattering result of this paper implies the scattering for defocusing cubic NLS on $\mathbb{R}^2 \times \mathbb{T}^2$, see \cite{Z1}. Such relations between NLS on waveguides and the resonant systems also appear in \cite{CHENG,CGZ, HP,YZ}.

\item
{\bf 1D cubic resonant system.}
We note that the scattering for the 1D cubic resonant system is established in \cite{YZ} so it is natural to consider the 2D analogue as in the current paper. 2D model is heuristically harder which causes differences and new difficulties. We will discuss more about it in the next subsection.\vspace{3mm}

\end{enumerate}

Since the resonant system is tightly related to \textbf{NLS on waveguides} problems, we also give an overview for the research line of `long time dynamics of NLS on waveguides' as below. (We will also make a few further remarks in the Appendix.) \vspace{3mm}

Generally, well-posedness theory and long time behavior of NLS is a hot topic in the area of dispersive evolution equations and has been studied widely in recent decades. Naturally, the Euclidean case is first treated and the theory at least in the defocusing setting has been well established. We refer to \cite{CKSTT,D3,KM1} for some important  Euclidean results. Moreover, we refer to \cite{CZZ,CGZ,HP,HTT1,HTT2,IPT3,IPRT3,KV1,TV1,TV2,YYZ,Z1,Z2,Z3} with regard to the tori case and the waveguide case. We may roughly think of the waveguide case as the ``intermediate point" between the Euclidean case and the tori case since the waveguide manifold is the product of the Euclidean space and the tori. The techniques used in Euclidean and tori settings are frequently combined and applied to the waveguides problems.\vspace{3mm}

At last, we refer to \cite{Bourgain1,C2,D3,Benbook,KM1,Taobook} and the reference therein for some classical results and theories on `Long time dynamics for NLS on Euclidean spaces'.
\subsection{Comparison with the 1D analogue}
In this section, we compare the 1D analogue (studied in \cite{YZ}) with the 2D model \eqref{maineq}. We will see that the proof for the 2D model is very similar to the 1D analogue in the following sections. However, the method for the 1D analogue can not directly adapted to the 2D case and the 2D model is essentially harder to deal with. We explain these points as below.\vspace{3mm}

1.\emph{More flexible directions.}\vspace{3mm}

We consider the resonant relation: $(j,j_1,j_2,j_3)$ satisfying $j=j_1-j_2+j_3$ and $|j|^2=|j_1|^2-|j_2|^2+|j_3|^2$. For the 1D case ($j,j_1,j_2,j_3\in \mathbb{Z}$), noticing that the variables are scalars instead of vectors, the resonant relation directly implies
\begin{equation}
    j=j_1,j_2=j_3 \quad \textmd{ or }\quad j=j_3,j_2=j_1.
\end{equation}
We can see if we fix $j$ and $j_2$ then the other two variables will be determined. Thus roughly speaking, the number of flexible (free) directions is two.

However, for the 2D case ($j,j_1,j_2,j_3\in \mathbb{Z}^2$), the situation would be quite different. The variables are vectors instead of scalars. Applying $j=j_1-j_2+j_3$ to $|j|^2=|j_1|^2-|j_2|^2+|j_3|^2$, we have
\begin{equation}
 j=j_1,j_2=j_3 \quad \textmd{ or }\quad (j-j_1) \cdot (j_1-j_2)=0.
\end{equation}
For this case, if we fix $j,j_1$ ($j \neq j_1$), $j_2$ can not be uniquely determined since there are infinite vectors which are perpendicular to a fixed nonzero vector. We can say the number of flexible directions is more than two (`between two and three') in some sense. This heuristically implies that the cubic resonances are more complicated for the 2D case. It is very reasonable that, with the increase of the dimension, the resonances become more complicated. 1D is a special case. \vspace{3mm}

2.\emph{The failure of $l^2$-estimate}
\vspace{3mm}

Another significant technical obstacle is the failure of $l^2$-estimate for the 2D case.  We denote $\vec{F}(u)$ to be the nonlinearity $\sum_{(j_1,j_2,j_3)\in R(j)} u_{j_{1}}\bar{u}_{j_{2}}u_{j_{3}}$. Here we note that the $l^2$ is the discrete $L^2$-norm of a sequence ($\|\vec{u}\|_{l^2}=(\sum_{j} |u_j|^2)^{\frac{1}{2}}$) and similarly $h^{a}$ is the discrete Sobolev norm ($||\vec{u}||_{h^a}=(\sum_{j} \langle 1+|j|^2 \rangle ^{a} |u_j|^2)^{\frac{1}{2}}$).

For the 1D case, one can prove the $l^2$-estimate for the cubic resonance in the following sense (see \cite{YZ})
\begin{equation}
\|\vec{F}(u)\|_{l^2}   \lesssim \|\vec{u}\|^3_{l^2}.
\end{equation}

However, for the 2D case, the above estimate does not work, which causes much trouble for proving large data scattering for \eqref{maineq}. We will explain it explicitly in Section 3. We note that the proof for Theorem \ref{mainthm} will be much simpler, i.e. almost follows line to line as the 1D case (\cite{YZ}) if the $l^2$-estimate holds.\vspace{3mm}

3.\emph{Weaker estimates and weak symmetries}\vspace{3mm}

Because of `the failure of $l^2$-estimate', we can not reduce the estimates to $l^2$-level as in \cite{YZ}. One main advantage of $l^2$-level type estimates is the good symmetric property. For instance, one can easily show
\begin{equation}
    \Im ( \sum_{(j,j_1,j_2,j_3) \in \mathcal{R}}u_j \bar{u}_{j_1}u_{j_2}\bar{u}_{j_3} )=0.
\end{equation}
using symmetric property by rotating the indices. If one considers $h^{\alpha}$-level estimate ($\alpha>0$), it would be harder to use the symmetry to deal with a term like this:
\begin{equation}
     \sum_{(j,j_1,j_2,j_3) \in \mathcal{R}}\langle j \rangle^{\alpha} u_j \bar{u}_{j_1}u_{j_2}\bar{u}_{j_3} .
\end{equation}
We will explain it more in Section 3 and Section 4.
\subsection{Overview of the strategy}

First of all, following the standard arguments, by using Strichartz estimates, one can get the local well-posedness, small data scattering and stability theory for the resonant system \eqref{maineq}. We summarize these results as follows. (See Appendix of \cite{Z1}. This part is also similar to Section 2 of \cite{YZ}.)
\begin{theorem}\label{mainloc} The resonant system \eqref{maineq} has the following properties:
\begin{enumerate}
  \item (Local well-posedness) Suppose that $\|\vec{u}_{0}\|_{L^2 h^1} \le E$,  then the resonant system \eqref{maineq} has a unique strong solution $\vec{u}\in C_t^0\left((-T,T);L^{2}h^{1}\right)\bigcap L_{t}^{4}L_{x}^{4}h^{1}\left((-T,T)\times \mathbb{R}^2\times \mathbb{Z}^2\right)$  for some $T > 0$, satisfying $\vec{u}(0)=\vec{u}_{0}$;
  \item (Small data scattering) There exists sufficiently small $\delta > 0$, if $\|\vec{u}_{0}\|_{L^2 h^1} \le \delta$,  then \eqref{maineq} has an unique global solution
$\vec{u}\in L_{t}^{\infty}L_{x}^{2}h^{1}\left(\mathbb{R}\times \mathbb{R}^2\times \mathbb{Z}^2\right)\bigcap L_{t}^{4}L_{x}^{4}h^{1}\left(\mathbb{R}\times \mathbb{R}^2\times \mathbb{Z}^2\right)$ with initial data $\vec{u}(0)=\vec{u}_{0}$, Moreover, there exist $\vec{u}^{\pm} \in L_{x}^{2}h^{1}(\mathbb{R}^2 \times \mathbb{Z}^2)$ such that
\begin{equation} \label{eq1.3}
\|\vec{u}(t)- e^{it\Delta} \vec{u}^{\pm}\|_{L_{x}^{2}h^{1}} \to 0, \ \text{ as } t\to \pm \infty.
\end{equation}
\item (Stability) For $a\in \{+1,0\}$, let I be a compact interval and $\vec{e}=\{e_{j}\}_{j\in \mathbb{Z}^2}$, $e_{j}=i\partial_t u_j + \Delta u_j - \sum\limits_{\mathcal{R}(j)} u_{j_1} \bar{u}_{j_2} u_{j_3}$, assume $\|\vec{u}\|_{L_{t}^{4}L_{x}^{4}h^{a}\left(I\times \mathbb{R}^2\times \mathbb{Z}^2\right)}\leq A$ for some $A>0$,  then for $\forall\epsilon>0,\exists \delta>0$,  such that if $\|\vec{e}\|_{L_{t}^{\frac{4}{3}}L_{x}^{\frac{4}{3}}h^{a}\left(I\times \mathbb{R}^2\times \mathbb{Z}^2\right)}\leq\delta,\|\vec{u}(t_{0})-\vec{v}_{0}\|_{L^{2}h^{1}}\leq\delta,$
then the resonant system \eqref{maineq} has a solution $\vec{v}\in L_{t}^{\infty}L_{x}^{2}h^{1}(I\times \mathbb{R}^2\times \mathbb{Z}^2)\cap L_{t}^{4}L_{x}^{4}h^{a}(I\times \mathbb{R}^2\times \mathbb{Z}^2)$ with initial data $\vec{v}(t_{0})=\vec{v}_0$, Moreover,
$$\|\vec{u}-\vec{v}\|_{L_{t}^{4}L_{x}^{4}h^{a}\left(I\times \mathbb{R}^2\times \mathbb{Z}^2\right)} +\|\vec{u}-\vec{v}\|_{L_{t}^{\infty}L_{x}^{2}h^{1}\left(I\times \mathbb{R}^2\times \mathbb{Z}^2\right)}\leq\epsilon.$$
\end{enumerate}
\end{theorem}

The real task is to generalize small data scattering result to large data scattering result. We note that, for classical NLS problems, this step is also highly nontrivial. We will use the standard concentration compactness/rigidity method established in Kenig-Merle \cite{KM,KM1}. In particular, we will apply Dodson's scheme \cite{D} which deals with the scattering for 2d mass-critical NLS. As shown above, the local theory and the small data scattering are standard. Moreover, the scattering norm is $L^4_{t,x}h^1$, whose finiteness implies scattering. We will see that we can reduce it to a smaller norm $L^4_{t,x}l^2$ (see Section 3.2). It suffices to show $\|u\|_{L^4_{t,x}l^2}\lesssim \infty$. Since the scheme is similar to the 1D analogue (see \cite{YZ}), we emphasize the new observations to fix the new difficulties as discussed in Section 1.3.\vspace{3mm}

1.\emph{Weaker estimates and the reduction of the scattering norm}\vspace{3mm}

Though the $l^2$-estimate does not hold for the 2D case, we can establish weaker estimates like (for some positive $\delta_1,\delta_2$ and $0<\beta<1$),
\begin{equation}
           \|\vec{F}(u)\|_{l^2} \lesssim \|\vec{u}\|_{l^2}\|\vec{u}\|^{2}_{h^{\beta}}.
\end{equation}
\begin{equation}
    \|\vec{F}(u)\|_{l^2} \lesssim \|\vec{u}\|^{\delta_1}_{l^2}\|\vec{u}\|^{3-\delta_1}_{h^1},
\end{equation}
and
\begin{equation}
           \|\vec{F}(u)\|_{h^1} \lesssim \|\vec{u}\|^{\delta_2}_{l^2}\|\vec{u}\|^{3-\delta_2}_{h^1},
\end{equation}
which are still helpful for us. (See Lemma \ref{es:resonant})\vspace{3mm}

Via the above estimates, we can reduce the scattering norm from the spacetime norm $L^4_{t,x}h^1$ to a smaller one $L^4_{t,x}l^2$ using the standard bootstrap argument. See Section 3 for more details. These estimates also help us in the other steps. (See Section 4 to Section 6.)\vspace{3mm}

2.\emph{Observations of the symmetries of the cubic resonances}\vspace{3mm}

As mentioned in Section 1.3, we can not do the estimates in the $l^2$ level as the 1D case due to the lack of $l^2$ estimate. As a consequence, we need to face regularities (in the discrete sense) in some estimates, which causes some trouble, i.e. destroying the symmetry). As we will show in Section 4, Section 5 and Section 6, we will use the $h^1$ level estimates instead of $l^2$ level estimates. Though the symmetric properties are not as good as the $l^2$ level case, we can still investigate some symmetric properties (see Section 4) and we are lucky enough to handle this model eventually.\vspace{3mm}

3.\emph{Modifications of the proof of the 1D analogue}\vspace{3mm}

Once we have the local theory and the nonlinear estimates for the cubic resonances, we will prove scattering for \eqref{maineq} via concentration compactness method using as in \cite{D,YZ}. To be more precise, long time Strichartz and frequency localized interaction Morawetz estimate are used. Since the nonlinear estimates are different and the estimates are of $h^1$-level, it causes some nontrivial changes in the proof. Thus we need to modify the proof of the 1D analogue \cite{YZ} suitably. See Section 5 and Section 6 for more details. We will state the arguments which follows in a natural way as in \cite{YZ} without proofs and give proofs for those where nontrivial changes happen.\vspace{3mm}

We will see more about the differences from the one discrete dimensional case in the following sections.\vspace{3mm}

\subsection{Organization of the rest of this paper}
In Section 2, we discuss the notations and the function spaces; in Section 3, we discuss the discrete estimates for the resonances, which may have its own interests; in Section 4, we discuss some useful observations and the preparations for the long time Strichartz estimate; in Section 5, we discuss the long time Strichartz estimate and Frequency localized interaction Morawetz estimate; in Section 6, we prove the rigidity theorem which concludes the main theorem; in Section 7 (Appendix), we make a few further remarks and include another proof for the discrete resonant nonlinearity.
\section{Preliminaries}
In this section, we discuss the notations and the function spaces.
\subsection{Notations}
Throughout this paper, we use $C$ to denote  the universal constant and $C$ may change line by line. We say $A\lesssim B$, if $A\leq CB$. We say $A\sim B$ if $A\lesssim B$ and $B\lesssim A$. We also use notation $C_{B}$ to denote a constant depends on $B$. We use usual $L^{p}$ spaces and Sobolev spaces $H^{s}$. Moreover, we write $\langle x \rangle= (1+|x|^2)^{\f{1}{2}}$, and $p'$ for the dual index of $p \in (1,+\infty)$ in the sense that $\frac{1}{p^{'}}+\frac{1}{p}=1$.

We regularly refer to the composed spacetime norms
\begin{equation}
    ||u||_{L^p_tL^q_z(I_t \times \mathbb{R}^m\times \mathbb{T}^n)}=\left(\int_{I_t}\left(\int_{\mathbb{R}^m\times \mathbb{T}^n} |u(t,z)|^q \, dz \right)^{\frac{p}{q}}  \, dt\right)^{\frac{1}{p}}.
\end{equation}
For $\vec{\phi}=\{\phi_p \}_{p \in \mathbb{Z}^2}$ a sequence of real-variable functions, we let
\begin{equation}
H^{s_2}h^{s_1}:=\{\vec{\phi}=\{ \phi_{p} \} : ||\vec{\phi}||_{H^{s_2}h^{s_1}}^2=\sum\limits_{p \in \mathbb{Z}^2} \langle p \rangle^{2s_1}||\phi_p||_{H^{s_2}}^{2} < +\infty \}.
\end{equation}
We note that $h^{0}=l^2$. Similarly, one can define norm $L^4_{t,x}h^1$ (taking the discrete norm $h^1$ first and then taking the spacetime norm).
\subsection{Function spaces}
In this section, we give the definition of $U^p_{\triangle}(h^{a})$,  $V^p_{\triangle}(h^{b})$, $a,b\in\{-1,0,1\}$ spaces and then prove corresponding bilinear Strichartz estimate in such spaces. These spaces are introduced in \cite{HTT1,HTT2} and have been widely used in papers such as \cite{D1,D,HP,IPT3,IPRT3}. We extend the setting to the system case as \cite{CGZ,YZ}.

\begin{definition} Let $1\leq p <\infty $.  Then $U^p_{\triangle}(h^{a})$, $a\in\{-1,0,1\}$ is an atomic space, where atoms are piecewise solutions
to the linear equation
$$\vec{u}=\sum\limits_{k}\chi_{[t_k,t_{k+1}]}(t)e^{it\triangle}\vec{u}_k(x),\,\,
\sum\limits_{k}\|\vec{u}_k(x)\|_{L^2h^{a}}^p=1.$$
We define $\|\cdot\|_{U^p_{\triangle}(h^{a})}$ as
$$\|\vec{u}\|_{U^p_{\triangle}(h^{a})}:=\inf\left\{\sum\limits_{\lambda}|c_{\lambda}|:\vec{u} =\sum_{\lambda}c_{\lambda}\vec{u}^{\lambda},\,\vec{u}^{\lambda} \,are\, U^p_{\triangle}(h^{a}) \,atoms\right\}.$$\\
Let $DU^p_{\triangle}(h^{a})$ be the space
$$DU^p_{\triangle}(h^{a})=\{(i\partial_{t}+\triangle)\vec{u}:\vec{u}\in U^p_{\triangle}(h^{a})\}$$
and the norm is
$$\|(i\partial_{t}+\triangle)\vec{u}\|_{DU^p_{\triangle}(h^{a})} :=\left\|\int_{0}^{t}e^{i(t-s)\triangle}(i\partial_{s}\vec{u}
+\triangle\vec{u})(s)ds\right\|_{U^p_{\triangle}(h^{a})}.$$
\end{definition}

\begin{definition}Let $1\leq p <\infty $.  Then we define $V^p_{\triangle}(h^{b})$, $b\in\{-1,0,1\}$ as the space of right continuous functions $\vec{v}\in L_{t}^{\infty}L_{x}^{2}h^{b}$ such that
$$\|\vec{v}\|_{V^p_{\triangle}(h^{b})}^p:=\|\vec{v}\|_{L_{t}^{\infty}L_{x}^{2}h^{b}}^p+\sup_{\{t_k\}\nearrow}\sum_{k}\|
e^{-it_{k+1}\triangle}\vec{v}(t_{k+1})-e^{-it_{k}\triangle}\vec{v}(t_{k})\|_{L^2h^{b}}^{p}<\infty,$$ 
where the supremum is taken over increasing sequences $t_{k}$.
\end{definition}

We collect some useful properties about $U^p_{\triangle}(h^{a})$,  $V^p_{\triangle}(h^{b})$, $a,b\in\{-1,0,1\}$ spaces below.
\begin{proposition}\label{pr-y5.3}
$U^p_{\triangle}(h^{a})$,  $V^p_{\triangle}(h^{b})$ space has the following properties:
\begin{enumerate}
  \item $U^p_{\triangle}(h^{a})$,  $V^p_{\triangle}(h^{b})$ is a Banach space.
  \item $U^p_{\triangle}(h^{a}) \subset V^p_{\triangle}(h^{a}) \subset U^q_{\triangle}(h^{a})$,  $1<p<q<\infty$.
  \item $(DU^p_{\triangle}(h^{a}))^{*}=V^{p'}_{\triangle}(h^{-a}),\frac{1}{p}+\frac{1}{p'}=1.$ and $1<p<\infty$.
  \item These spaces are also closed under truncation in time.$$\chi_{I}:U^p_{\triangle}(h^{a})\rightarrow U^p_{\triangle}(h^{a});\,\,\,\,\chi_{I}:V^p_{\triangle}(h^{b})\rightarrow V^p_{\triangle}(h^{b}).$$
  \item  Suppose $J = I_1 \cup I_2 , I_1 = [a,b], I_2 = [b,c], a\leq b\leq c$, then
  $$\|\vec{u}\|_{U^p_{\triangle}(h^{a};J)}^{p}\leq\|\vec{u}\|_{U^p_{\triangle}(h^{a};I_1)}^{p}+
  \|\vec{u}\|_{U^p_{\triangle}(h^{a};I_2)}^{p},$$
  $$\|\vec{u}\|_{U^p_{\triangle}(h^{a};I_1)}\leq \|\vec{u}\|_{U^p_{\triangle}(h^{a};J)}.$$
  \item $\|\vec{u}\|_{L^p_t L^q_x h^{a}}+\|\vec{u}\|_{L^{\infty}_t L^2_x h^{a}}\lesssim \|\vec{u}\|_{U^p_{\triangle}(h^{a})}$, $(p,q)$ is an admissible pair, $p>2$.  i.e. $\frac{1}{p}+\frac{1}{q}=\frac{1}{2}.$
  \item There is the easy estimate
      $$\|\vec{u}\|_{U^p_{\triangle}(h^{a})}\lesssim\|\vec{u}(0)\|_{L^2h^{a}}+
      \|(i\partial_{t}+\triangle)\vec{u}\|_{DU^p_{\triangle}(h^{a})}.$$
\end{enumerate}
\end{proposition}
\begin{proof}The proof is standard. the reader can refer to Section 4 in \cite{YZ} so we omit it here.
\end{proof}    In the following we will use  $U^p_{\triangle}(h^{1})$ spaces to define long time Strichartz estimate norm and perform  the arguments in the section 5 of \cite{D} on $h^1$ level at the discrete direction of ``$j$". Therefore we will focus on the properties of the spaces involved $h^1$ norm at the discrete direction of ``$j$" below.
\begin{lemma}\label{le-y4.5}
Suppose $J=\cup ^k_{m=1}J^m$, where $J^m$ are consecutive intervals, $J^m=[a_m,b_m]$, $a_{m+1}=b_m$. Also suppose that $\vec{F}=\{F_j\}_{j\in\mathbb{Z}}\in L ^1_t L^2_x h^{1} (J \times\mathbb{R}^2)$ (however our bound will not depend on $\|\vec{F}\|_{L ^1_t L^2_x h^{1}}$.) Then
for any $t_0 \in J$,
\begin{equation}\label{eq-y4.1'}
\begin{split}
  \big\|\int^t_{t_0} e^{i(t-\tau)\triangle}\vec{F}(\tau)d\tau \big\|_{U^2_{\triangle}(h^{1};J\times \mathbb{R}^2)} &\lesssim \sum^{k}_{m=1} \big\| \int_{J^m} e^{-i\tau\triangle}\vec{F}(\tau)d\tau \big\|_{L^2_x h^{1}} + \big[\sum^{k}_{m=1}(\|\vec{F}\|_{DU^{2}_{\triangle}(h^{1};J^m\times \mathbb{R}^2)})^2 \big]^{1/2}.
\end{split}
\end{equation}Where $$\|\vec{F}\|_{DU^{2}_{\triangle}(h^{1};J^m\times \mathbb{R}^2)}:=\sup_{\|\vec{v}\|_{V^{2}_{\triangle}(h^{-1};J^m\times \mathbb{R}^2)}=1}\int_{J^m}\sum_{j\in \mathbb{Z}}F_j(\tau)v_j(\tau) d\tau.$$
\end{lemma}
\begin{proof}
The proof is the same as Lemma 3.4 in \cite{D}. So we won't repeat it here.
\end{proof}
\begin{proposition}[bilinear Strichartz estimate]\label{pr-y5.4}
$\frac{1}{p}+\frac{1}{q}=1,\vec{u}_{0}=\{u_{0,j}\}_{j\in\mathbb{Z}}, \vec{v}_0=\{v_{0,j}\}_{j\in\mathbb{Z}}$. Assume  $\hat{u}_{0,j}$ is supported on $|\xi|\sim N$ and $\hat{v}_{0,j}$ is supported on $|\xi|\sim M$ for $j\in\mathbb{Z}$. If $M \ll N$, then
\begin{equation}\label{eq-y4.2'}
  \left\|\|e^{it\triangle}\vec{u}_{0}\|_{h^{1}}\cdot\|e^{\pm it\triangle}\vec{v}_0\|_{h^{1}}\right\|_{L_t^p L_x^q(\mathbb{R}\times\mathbb{R}^2)}\lesssim\big(\frac{M}{N}\big)^{\frac{1}{p}}\|\vec{u}_{0}\|_{L_x^2h^{1}}
  \|\vec{v}_0\|_{L_x^2h^{1}}.
\end{equation}
\end{proposition}
\begin{remark}\label{convolution}
\begin{enumerate}
  \item Suppose that $g(t,x-y)$ and $h(t,x-z)$ are convolution kernels with the bounds
\begin{equation}\label{condition1}
 \|\sup_{t\in \mathbb{R}}|g(t,x)|\|_{L^1(\mathbb{R}^2)}\lesssim1\quad and\quad \|\sup_{t\in \mathbb{R}}|h(t,x)|\|_{L^1(\mathbb{R}^2)}\lesssim1,
\end{equation}then under the hypothesis of Proposition \ref{pr-y5.4}, we have
\begin{equation}\label{coneq-y4.2'}
  \left\|\|g\ast e^{it\triangle}\vec{u}_{0}\|_{h^{1}}\cdot\|h\ast e^{\pm it\triangle}\vec{v}_0\|_{h^{1}}\right\|_{L_t^p L_x^q(\mathbb{R}\times\mathbb{R}^2)}\lesssim\big(\frac{M}{N}\big)^{\frac{1}{p}}\|\vec{u}_{0}\|_{L_x^2h^{1}}
  \|\vec{v}_0\|_{L_x^2h^{1}},
\end{equation}where $\|g\ast e^{it\triangle}\vec{u}_{0}\|_{h^{1}}=\big(\sum_j|g\ast e^{it\triangle}u_{0,j}|^2\big)^{1/2}$. The kernels of $P_{\xi(t),j}$, $P_{\xi(t),\leq j}$ and $P_{\xi(t),\geq j}$ all satisfy \eqref{condition1}.
  \item For $a,b\in\{-1,0,1\}$, we still have
  \begin{equation}\label{eq-yl4.2''}
  \left\|\|e^{it\triangle}\vec{u}_{0}\|_{h^{a}}\cdot\|e^{\pm it\triangle}\vec{v}_0\|_{h^{b}}\right\|_{L_t^p L_x^q(\mathbb{R}\times\mathbb{R}^2)}\lesssim\big(\frac{M}{N}\big)^{\frac{1}{p}}\|\vec{u}_{0}\|_{L_x^2h^{a}}
  \|\vec{v}_0\|_{L_x^2h^{b}}.
\end{equation}
\end{enumerate}
\end{remark}

\begin{proposition}\label{pr-y5.5}
 For $\vec{u}=\{u_j\}_{j\in\mathbb{Z}},\vec{v}=\{v_j\}_{j\in\mathbb{Z}}$, we assume $supp\ \ \hat{u}_j\subset\{\xi:|\xi|\sim N\},supp\ \ \hat{v}_j\subset\{\xi:|\xi|\sim M\}$. If $M\ll  N$, we have
\begin{equation}
\left\|\|\vec{u}\|_{h^{1}}\cdot\|\vec{v}\|_{h^{1}}\right\|_{L^p_t L^q_x}\lesssim\big(\frac{M}{N}\big)^{\frac{1}{p}}\|\vec{u}\|_{U^p_{\triangle}(h^{1})} \|\vec{v}\|_{U^p_{\triangle}(h^{1})},
\end{equation}where $\frac{1}{p}+\frac{1}{q}=1$.
\end{proposition}
\begin{proof}
We just take $U^p_{\triangle}(h^{1})$ and $U^p_{\triangle}(h^{1})$ atoms into consideration .  Let
$$\vec{u}=\sum_k \chi_{[t_k,t_{k+1}]}e^{it\triangle}\vec{u}_{k}, \sum\limits_{k}\|\vec{u}_k(x)\|_{L_x^2h^{1}}^p=1; \vec{v}=\sum_{k'} \chi_{[t_{k'},t_{k'+1}]}e^{it\triangle}\vec{v}_{k'}, \sum\limits_{k'}\|\vec{v}_{k'}(x)\|_{L_x^2h^{1}}^p=1.$$
Without loss of generality, we may assume $\|\vec{u}\|_{U^p_{\triangle}(h^{1})}=1, \|\vec{v}\|_{U^p_{\triangle}(h^{1})}=1$.  It suffices to show
$$\left\|\|\vec{u}\|_{h^{1}}\cdot\|\vec{v}\|_{h^{1}}\right\|_{L^p_t L^q_x}\lesssim \big(\frac{M}{N}\big)^{\frac{1}{p}}.$$
By Minkowski inequality and Proposition \ref{pr-y5.4}, we have
\begin{align*}
  \left\|\|\vec{u}\|_{h^{1}}\cdot\|\vec{v}\|_{h^{1}}\right\|^p_{L^p_t L^q_x}=&\left\|\|\sum_k \chi_{[t_k,t_{k+1}]}e^{it\triangle}\vec{u}_{k}\|_{h^{1}}\cdot\|\sum_{k'} \chi_{[t_{k'},t_{k'+1}]}e^{it\triangle}\vec{v}_{k'}\|_{h^{1}}\right\|^p_{L^p_t L^q_x} \\
  \leq&\left\|(\sum_{k}\chi_{[t_k,t_{k+1}]}\|e^{it\triangle}\vec{u}_{k}\|_{h^{1}})\cdot(\sum_{k'} \chi_{[t_{k'},t_{k'+1}]}\|e^{it\triangle}\vec{v}_{k'}\|_{h^{1}})\right\|^p_{L^p_t L^q_x}\\
  \leq&\int_{\mathbb{R}}\left(\sum_{k,k'}\chi_{[t_k,t_{k+1}]}\chi_{[t_{k'},t_{k'+1}]} \left\|\|e^{it\triangle}\vec{u}_{k}\|_{h^{1}} \|e^{it\triangle}\vec{v}_{k'}\|_{h^{1}}\right\|_{L_x^q}\right)^p(t)dt\\
  \leq&\sum_k\sum_{k'}\int_{[t_{k'},t_{k'+1}]\cap[t_{k},t_{k+1}]} \left(\left\|\|e^{it\triangle}\vec{u}_{k}\|_{h^{1}} \|e^{it\triangle}\vec{v}_{k'}\|_{h^{1}}\right\|_{L_x^q}\right)^p(t)dt\\
  \leq&\frac{M}{N}\sum_{k}\sum_{k'}\|\vec{u}_{k}\|^p_{L^2_x h^{1}}\|\vec{v}_{k'}\|^p_{L^2_x h^{1}}\\
  =&\frac{M}{N}.
\end{align*}Therefore,
$$\left\|\|\vec{u}\|_{h^{1}}\cdot\|\vec{v}\|_{h^{1}}\right\|_{L^p_t L^q_x}\lesssim \big(\frac{M}{N}\big)^{\frac{1}{p}}.$$
\end{proof}

\begin{remark}
For $a,b\in\{-1,0,1\}$, $\frac{1}{p}+\frac{1}{q}=1$,  we still have
\begin{equation}
\left\|\|\vec{u}\|_{h^{a}}\cdot\|\vec{v}\|_{h^{b}}\right\|_{L^p_t L^q_x}\lesssim\big(\frac{M}{N}\big)^{\frac{1}{p}}\|\vec{u}\|_{U^p_{\triangle}(h^{a})} \|\vec{v}\|_{U^p_{\triangle}(h^{b})}.
\end{equation}
\end{remark}

\section{Estimates for the cubic resonances}
We recall the cubic resonant system as follows \eqref{maineq}.
\begin{equation}\label{eq-main1.1}
\aligned
&(i\partial_t+\Delta_{x})u_j=\sum_{(j_1,j_2,j_3)\in R(j)} u_{j_{1}}\bar{u}_{j_{2}}u_{j_{3}}, \\
&R(j)=\{(j_1,j_2,j_3) \in (\mathbb{Z}^2)^3:j_1-j_2+j_3=j \quad and \quad  |j_1|^2-|j_2|^2+|j_3|^2=|j|^2 \}
\endaligned
\end{equation}
with unknown $\vec{u}=\{u_j \}_{j \in \mathbb{Z}^2}$.\vspace{3mm}

It is important to understand the resonant nonlinearity when one studies the cubic resonant system. To be more precise, we want to estimate the resonant nonlinearity. We note again that the above system corresponds to the two dimensional case since $j \in \mathbb{Z}^2$. The 1D case (\cite{YZ}) and the 2D case are quite different.\vspace{3mm}

We consider the discrete nonlinearity $\vec{F}(u)=\{ \sum_{(j_1,j_2,j_3)\in R(j)} u_{j_{1}}\bar{u}_{j_{2}}u_{j_{3}} \}_{j}$ first, which means $u_j$ are numbers in $\mathbb{C}$. (Later for the scattering problem of \eqref{maineq}, we consider $u_j$ to be functions from $\mathbb{R}^2$ to $\mathbb{C}$.) The study of this discrete model may have its own interests.\vspace{3mm}

One question is: Do the following two estimates hold?
\begin{equation}
    \|\vec{F}(u)\|_{l^2} \lesssim \|\vec{u}\|^3_{l^2},
\end{equation}
and
\begin{equation}
        \|\vec{F}(u)\|_{h^1} \lesssim \|\vec{u}\|^2_{l^2}\|\vec{u}\|_{h^1}.
\end{equation}
For one dimensional case, the answer is `Yes'. They do hold since this case is very special. (See \cite{YZ}).

Now we present some controls regarding the nonlinearity. We will explain why the estimates do not hold and how to obtain some weaker estimates which are still useful for studying our model \eqref{maineq}.

We briefly recall some notations.
Let $a:=\{a_{j}\}_{j\in \mathbb{Z}^{2}}$, $a_{j}\in \mathbb{C}$, and let

\begin{equation}
\begin{aligned}
&\|a\|_{l^{2}}^{2}:=\sum_{j\in \mathbb{Z}^{2}}|a_{j}|^{2},\\
&\|a\|_{h^{s}}^{2}:=\sum_{j\in \mathbb{Z}^{2}}\langle j \rangle^{2s}|a_{j}|^{2},
\end{aligned}
\end{equation}
and let
\begin{equation}
\begin{aligned}
&F(a):=\{f_{j}\}_{j\in \mathbb{Z}^{2}},\\
&f_{j}:=\sum_{(j_1,j_2,j_3)\in R(j)} a_{j_{1}}\bar{a}_{j_{2}}a_{j_{3}}.
\end{aligned}
\end{equation}

And it would be favorable if one has
\begin{equation}\label{eq: fav}
\|F(a)\|_{l^{2}}\lesssim \|a\|_{l^{2}}^{3}.
\end{equation}

Estimate of type \eqref{eq: fav} has close connection to Strichartz estimates for Schr\"odinger equations on $\mathbb{T}^{2}$.  We do an explanation here for the convenience of the readers, and we point out similar computations frequently appear in the literature of random data. (See \cite{Deng1,Deng2} and the reference therein.)

Given $a$, let
\begin{equation}\label{eq: e11}
v(t,x):=v_{a}(t,x)=\sum_{j}a_{j}e^{ijx}e^{ij^{2}t}.	
\end{equation}
Then we have $v(t,x)$ solves
\begin{equation}\label{eq: e12}
iv_{t}+\Delta v=0,
v(0,x)=v_{0}=\sum_{j}a_{j}e^{ijx}.
\end{equation}
In particular
\begin{equation}\label{eq: e13}
\|a\|_{l^{2}}=\|v_{0}\|_{l^{2}}, \quad \|a\|_{h^{s}}=\|v_{0}\|_{H^{s}}.	
\end{equation}

Meanwhile,
\begin{equation}\label{eq: e14}
\|F(a)\|_{l^{2}}:=\sup_{\|b\|_{l^{2}}=1}\sum_{(j_{1},j_{2},j_{3})\in R(j)}\bar{b}_{j}a_{j_{1}}	\bar{a}_{j_{2}}a_{j_{3}},
\end{equation}

We similarly define, for every $\|b\|_{l^{2}}=1$
\begin{equation}\label{eq: e15}
g(t,x):=g_{b}(t,x)=\sum_{j}b_{j}e^{ijx}e^{ij^{2}t}.
\end{equation}
And we have
\begin{equation}\label{eq: e16}
	\sum_{(j_{1},j_{2},j_{3})\in R(j)}\bar{b}_{j}a_{j_{1}}	\bar{a}_{j_{2}}a_{j_{3}}=(2\pi)^{-3}\int_{\mathbb{T}^{2}\times [0,2\pi]}. |v|^{2}v\bar{g}dxdt
\end{equation}

From this perspective, (plugging $g=v$), estimate \eqref{eq: fav} is equivalent to Strichartz type estimate
\begin{equation}\label{eq: favstri}
	\|e^{it\Delta_{\mathbb{T}^{2}}}\phi\|_{L_{t,x}^{4}([0,2\pi]\times \mathbb{T}^2)}\lesssim \|f\|_{L_{x}^{2}}.
\end{equation}

To the best of our knowledge, Estimate \eqref{eq: favstri} is unknown, and it is probably wrong, given its 1d parallel version (replacing $L^{4}$ by $L^{6}$) is known to be wrong. (See \cite{Taobook}).

Meanwhile, one can also rely the above connection to derive estimates regarding $f(a)$, with the help of linear/multi linear Strichartz estimates on the torus.

One has
\begin{lemma}\label{es:resonant}
For all $\beta>0$,
\begin{equation}\label{eq: l2es}
\|f(a)\|_{l^{2}}\lesssim \|a\|_{l^{2}}^{2}\|a\|_{h^{\beta}},
\end{equation}
\begin{equation}\label{eq: h1es}	
\|f(a)\|_{h^{1}}\lesssim \|a\|_{h^{1}}\|a\|_{h^{\beta}}\|a\|_{l^{2}}.
\end{equation}
\end{lemma}
\begin{proof}
Let $e^{it\Delta_{\mathbb{T}_{2}}}$ be the linear Schr\"odinger propagator on $\mathbb{T}^{2}$, and $P_{N}$ be Littlewood Paley projection at frequency $\sim N$, Let $N_{1}\geq N_{2}$ be dyadic integers.Let $v_{0}$be in $L_{x}^{2}(\mathbb{T})$.

\begin{equation}\label{eq: l2strichartz}
\|e^{it\Delta}P_{N}v_{0}\|_{L_{t,x,loc}^{4}}\lesssim_{\epsilon}N^{\beta}\|v_{0}\|_{2}
\end{equation}
and

\begin{equation}\label{eq: h1strichartz}
\|P_{N_{1}}e^{it\Delta_{\mathbb{T}^{2}}}v_{0}P_{N_{2}}e^{it\Delta_{\mathbb{T}^{2}}}v_{0}\|_{L_{t,x,loc}^{2}}\lesssim_{\epsilon}N_{2}^{\epsilon}\|v_{0}\|_{2}^{2}.
\end{equation}	
By the connection explained in \eqref{eq: e11}-\eqref{eq: favstri},
estimate \eqref{eq: l2es} follows from Strichartz type estimates \eqref{eq: l2strichartz}, and estimate \eqref{eq: h1es} follows from Strichartz type estimates, \eqref{eq: h1strichartz}.
\end{proof}

\begin{remark}
We include another proof for the above estimates with a certain range of $\beta>0$ in the appendix following the strategy in \cite{HP,YZ,Z1} (using number theory).
\end{remark}

\emph{Reduction of the scattering norm. } It is known that for \eqref{maineq}, scattering in $L^2_xh^1$ space is equivalent to the finiteness of $L^4_{t,x}h^1$ norm. We argue that it suffices to show the finiteness of $L^4_{t,x}l^2$ norm, which is weaker. That's good since we can use some symmetric properties. Also, naturally, a smaller quantity is more likely to be finite. This follows from \eqref{eq: h1es} and the trick of interval divisions.

Assuming $\|\vec{u}\|_{L^4_{t,x}l^2}<\infty$, Strichartz estimate gives
\begin{equation}
\aligned
  \|\vec{u}\|_{L^4_{t,x}h^1} &\lesssim \|\vec{u}_0\|_{L^2h^1}+\|\vec{F}(u)\|_{L_{t,x}^{\frac{4}{3}}h^1} \\
  &\lesssim \|\vec{u}_0\|_{L^2h^1}+\big\| \|\vec{u}\|^{\delta_2}_{l^2}\|\vec{u}\|^{3-\delta_2}_{h^1} \big\|_{L_{t,x}^{\frac{4}{3}}}\\
  &\lesssim \|\vec{u}_0\|_{L^2h^1}+ \|\vec{u}\|^{\delta_2}_{L^4_{t,x}l^2}\|\vec{u}\|^{3-\delta_2}_{L^4_{t,x}h^1}.
  \endaligned
\end{equation}
Continuity arguments give us the conclusion.\vspace{3mm}

\begin{remark}
If on a time interval $I$, we have the control for $L^4_{t,x}l^2$, then we can control $L^4_{t,x}h^1$ as well according to the nonlinear estimate. So on small intervals, we can reduce $h^1$ to $L^2$.

However, this observation is not enough to reduce everything to $l^2$ because of the failure of $l^2$-estimate, thus we still need to deal with the case with regularity as shown in the following sections.
\end{remark}

\section{Some useful observations and preparations}
\subsection{Some useful observations}
We discuss some observations on the symmetric property of the cubic resonances. For the 1D case, there is no regularity, i.e. everything is in $l^2$ norm, which gives us good symmetric property. For example, it is easy to show the following equality noticing the symmetry,
\begin{equation}
    \Im ( \sum_{(j,j_1,j_2,j_3) \in \mathcal{R}}u_j \bar{u}_{j_1}u_{j_2}\bar{u}_{j_3} )=0.
\end{equation}
We observe that if we consider $h^1$ regularity, we still have some symmetric property to `save' the 2D model \eqref{maineq} though it is not as good as the no-regularity case. For $h^{\epsilon}$ regularity case, there is no symmetry. Thus we will use $h^1$ regularity though $h^{\epsilon}$ regularity is smaller and closer to $l^2$. We state the observations as follows.

\textbf{1. First observation.}

\begin{align}
&\Im\left(\sum\limits_{j}\langle j\rangle^{2} u_{j} \sum\limits_{\mathcal{R}_{j}} \bar{u}_{j_{1}} u_{j_{2}} \bar{u}_{j_{3}}\right) \\
&=\Im\left(\sum\limits_{\left(j, j_{1}, j_{2}, j_{3}\right) \in \mathcal{R}}\langle j\rangle^{2} u_{j} \bar{u}_{j_{1}} u_{j_{2}} \bar{u}_{j_{3}}\right) \\
&=\Im\left(\sum\limits_{\left(j, j_{1}, j_{2}, j_{3}\right) \in \mathcal{R}}|j|^{2} u_{j} \bar{u}_{j_{1}} u_{j_{2}} \bar{u}_{j_{3}}\right) \\
&=\frac{1}{2} \Im\left( \sum\limits_{\left(j, j_{1}, j_{2}, j_{3}\right) \in \mathcal{R}}\left(|j|^{2}+\left|j_{2}\right|^{2}\right) u_{j} \bar{u}_{j_{1}} u_{j_{2}} \bar{u}_{j_{3}}\right) \\
&=\frac{1}{4} \Im\left( \sum\limits_{\left(j, j_{1}, j_{2}, j_{3}\right) \in \mathcal{R}}\left(|j|^{2}+\left|j_{1}\right|^{2}+\left|j_{2}\right|^{2}+\left|j_{3}\right|^{2}\right) u_{j} \bar{u}_{j_{1}} u_{j_{2}} \bar{u}_{j_{3}}\right) =0,
\end{align}
where
\begin{equation}
 \mathcal{R}_j=\left\{\left(j_{1}, j_{2}, j_{3}\right) \in\left(\mathbb{Z}^{2}\right)^{3}: j_{1}-j_{2}+j_{3}=j \text { and }\left|j_{1}\right|^{2}-\left|j_{2}\right|^{2}+\left|j_{3}\right|^{2}=|j|^{2}\right\},
\end{equation}
and
\begin{equation}
\mathcal{R}=\left\{\left(j,j_{1}, j_{2}, j_{3}\right) \in\left(\mathbb{Z}^{2}\right)^{4}: j_{1}-j_{2}+j_{3}=j \text { and }\left|j_{1}\right|^{2}-\left|j_{2}\right|^{2}+\left|j_{3}\right|^{2}=|j|^{2}\right\}.
\end{equation}

\textbf{2. Second observation.}
\begin{equation*}
\begin{split}
  & \Im\Big(\sum\limits_{\left(j, j_{1}, j_{2}, j_{3}\right) \in \mathcal{R}}\langle j\rangle^{2}[ (P_{\xi(t),\leq l_2}\bar{u}^h_{j})  u^l_{j_{1}} \bar{u}^l_{j_{2}} u^l_{j_{3}}+\bar{u}^l_{j}  (P_{\xi(t),\leq l_2}u^h_{j_{1}}) \bar{u}^l_{j_{2}} u^l_{j_{3}}\\
  &+\bar{u}^l_{j}  u^l_{j_{1}} (P_{\xi(t),\leq l_2}\bar{u}^h_{j_{2}}) u^l_{j_{3}}+\bar{u}^l_{j}  u^l_{j_{1}} \bar{u}^l_{j_{2}} (P_{\xi(t),\leq l_2}u^h_{j_{3}})]\Big) \\
  =&\Im\Big(\sum\limits_{\left(j, j_{1}, j_{2}, j_{3}\right) \in \mathcal{R}}(\langle j\rangle^{2}+\langle j_2\rangle^{2})[ (P_{\xi(t),\leq l_2}\bar{u}^h_{j})  u^l_{j_{1}} \bar{u}^l_{j_{2}} u^l_{j_{3}}+\bar{u}^l_{j}  (P_{\xi(t),\leq l_2}u^h_{j_{1}}) \bar{u}^l_{j_{2}} u^l_{j_{3}}]\Big)\\
  =&0. \\
\end{split}
\end{equation*}
This cancellation relation will be used in the estimate of the nonlinear part ``$F_{1,j}$" as to (6.27) in \cite{YZ}, when we replace ``$l^2$" summation of ``$h^1$" summation.

\textbf{3.Third observation.}
\begin{equation}
\begin{split}
  & \Im\Big(\sum\limits_{\left(j, j_{1}, j_{2}, j_{3}\right) \in \mathcal{R}}\langle j\rangle^{2}[ (P_{\xi(t),\leq l_2}\bar{u}^h_{j})  u^l_{j_{1}} \bar{u}^l_{j_{2}} (P_{\xi(t),\leq l_2}u^h_{j_{3}})+\bar{u}^l_{j}  (P_{\xi(t),\leq l_2}u^h_{j_{1}}) (P_{\xi(t),\leq l_2}\bar{u}^h_{j_{2}}) u^l_{j_{3}}]\Big) \\
  =&\Im\Big(\sum\limits_{\left(j, j_{1}, j_{2}, j_{3}\right) \in \mathcal{R}}(\langle j\rangle^{2}+\langle j_2\rangle^{2})[ (P_{\xi(t),\leq l_2}\bar{u}^h_{j})  u^l_{j_{1}} \bar{u}^l_{j_{2}} (P_{\xi(t),\leq l_2}u^h_{j_{3}})]\Big)\\
  =&0. \\
\end{split}
\end{equation}

\subsection{Some preparations for the Long time Strichartz estimate.}
In the next section, we will use the above observations to give the crucial estimate: Long time Strichartz estimate. Before that, we will first do some preparations. 

Let $\phi \in C^{\infty}_0(\mathbb{R}^2)$ be a radial, decreasing function
\begin{equation*}\phi(x):=
\left\{
  \begin{array}{ll}
    1, & \hbox{$|x|\leq 1$,} \\
    0, & \hbox{$|x|>2$.}
  \end{array}
\right.
\end{equation*}
Define the partition of unity
\begin{equation*}
1=\phi(x)+\sum_{j=1}^{\infty}[\phi(2^{-j}x)-\phi(2^{-j+1}x)]=:\psi_0(x)+\sum_{j=1}^{\infty}\psi_j(x).
\end{equation*}
For any integer $j\geq 0$, let
\begin{equation*}
  P_{j}f= \mathcal{F}^{-1}(\psi_j(\xi)\hat{f}(\xi))=\int K_j(x-y)f(y)dy,
\end{equation*}
where $K_j$ is an $L^1$ -kernel. When $j$ is an integer less than zero, let $P_j f = 0$. Finally let
\begin{equation*}
  P_{j_1\leq\cdot\leq j_2}f=\sum_{j_1\leq\j\leq j_2}P_j f.
\end{equation*}
We also define the frequency truncation
\begin{equation*}
  P_{\leq j}f= \mathcal{F}^{-1}(\phi(2^{-j}\xi)\hat{f}(\xi)).
\end{equation*}
Let $\xi_0\in \mathbb{R}^2$, then define $P_{\xi_0,j}u=e^{ix\cdot\xi_0}P_j(e^{-ix\cdot\xi_0})u$ and $P_{\xi_0,j}\vec{u}=\{P_{\xi_0,j}u_k\}_{k\in\mathbb{Z}^2}$. Similarly we can define $P_{\xi_0,\leq j}\vec{u}$ and $P_{\xi_0,\geq j}\vec{u}$. Finally we point out  $P_{\xi_0,j}\vec{\bar{u}}=\{\overline{P_{\xi_0,j}u_k}\}_{k\in\mathbb{Z}^2}$.

Record $\mathbf{F}(\vec{u}):=\vec{F}(\vec{u}):=\{F_j(\vec{u})\}_{j\in\mathbb{Z}^2}:=
\{\sum\limits_{\mathcal{R}(j)}u_{j_1}\bar{u}_{j_2}u_{j_3}\}_{j\in\mathbb{Z}^2}=:\sum^{\rightarrow}\limits_{\mathcal{R}(j)}u_{j_1}\bar{u}_{j_2}u_{j_3}$,
and  $u^l_{\sharp}:=P_{\xi(t),\leq l_2-5}u_{\sharp}$, $u^{h}_{\sharp}=u_{\sharp}- u^l_{\sharp}$, for $\sharp\in\{j,j_1,j_2,j_3,j',\cdots\}$.

We let $G$ be the  group  generated by phase rotations, Galilean transforms, translations and dilations. We let $G\backslash H^{1}h^{1}(\mathbb{R}^2\times\mathbb{Z}^2)$ be the modulo space of $G$-orbits $G\vec{f}:=\{g\vec{f}:g\in G\}$, $g_{\theta,\xi_0,x_0,\lambda}\vec{f}(x):=\frac{1}{\lambda}e^{i\theta}e^{ix\xi_0}\vec{f}(\frac{x-x_0}{\lambda})$ endowed with the usual quotient topology. Define
$$(T_{g_{\theta,\xi_0,x_0,\lambda}}\vec{u})(t,x):=\frac{1}{\lambda}e^{i\theta}e^{ix\cdot\xi_0}e^{-it|\xi_0|^2}
\vec{u}\left(\frac{t}{\lambda^2},\frac{x-x_0-2\xi_0t}{\lambda}\right),$$
 then the map $g\mapsto T_g$ is a group action of $G$.

 Following the argument in the proof of  Proposition 3.1 in \cite{CHENG}, we can similarly have
 \begin{proposition}[Linear profile decomposition in $L_x^2 h^1(\mathbb{R}^2 \times \mathbb{Z}^2)$]  \label{pr6.640}
Let $\{\vec{u}_{n}\}$ be a bounded sequence in $L_x^2 h^1(\mathbb{R}^2 \times \mathbb{Z}^2)$. Then (after passing to a subsequence if necessary) there
exists $J\in  \{0,1, \cdots \} \cup \{\infty\}$, functions $\{\vec{\phi}^{j}\}_{j=1}^{J} \subseteq L_x^2 h^1$, group elements $\{g_n^j\}_{j=1}^{J} \subseteq G$, $g_n^j:=g_{\theta^{j}_{n},\xi_{n}^{j},x_{n}^j,\lambda_{n}^j}$,
and times $\{t_n^j\}_{j=1}^{J} \subseteq \mathbb{R}$ so that defining $\vec{w}_{n}^J$ by
\begin{align*}
\vec{u}_n(x) = &  \sum_{j=1}^J g_n^j e^{it_n^j \Delta_{\mathbb{R}^2}} \vec{\phi}^j + \vec{w}_n^J(x) \\
         := & \sum_{j=1}^J \frac1{\lambda_n^j} e^{ix\xi_n^j} (e^{it_n^j \Delta_{\mathbb{R}^2} } \vec{\phi}^j)\left(\frac{x-x_n^j}{\lambda_n^j} \right) + \vec{w}_n^J(x ),
\end{align*}
we have the following properties:
\begin{align*}
\limsup_{n\to \infty} \|e^{it\Delta_{\mathbb{R}^2}} \vec{w}_n^J\|_{L_{t,x}^4 l^4(\mathbb{R} \times \mathbb{R}^2 \times \mathbb{Z}^2 )} \to 0, \ \text{ as } J\to \infty, \\
e^{-it_n^j \Delta_{\mathbb{R}^2}} (g_n^j)^{-1} \vec{w}_n^J \rightharpoonup 0  \text{ in } L_x^2 h^1,  \text{ as  } n\to \infty,\text{ for each }  j\le J,\\
\sup_{J} \lim_{n\to \infty} \left(\|\vec{u}_n\|_{L_x^2 h^1}^2 - \sum_{j=1}^J \|\vec{\phi}^j\|_{L_x^2 h^1}^2 - \|\vec{w}_n^J\|_{L_x^2 h^1}^2\right) = 0,
\end{align*}
and lastly, for $j\ne j'$, and $n\to \infty$,
\begin{align*}
\frac{\lambda_n^j}{\lambda_n^{j'}} + \frac{\lambda_n^{j'}}{\lambda_n^j} + \lambda_n^j \lambda_n^{j'} |\xi_n^j - \xi_n^{j'}|^2 + \frac{|x_n^j-x_n^{j'}|^2 } {\lambda_n^j \lambda_n^{j'}}  + \frac{|(\lambda_n^j)^2 t_n^j -(\lambda_n^{j'})^2 t_n^{j'}|}{\lambda_n^j \lambda_n^{j'}} \to \infty.
\end{align*}%
\end{proposition}

\begin{remark}\label{re-y4.2}
By using interpolation, the H\"older inequality and Proposition \ref{pr6.640}, for $0<\epsilon_0\leq 1$,  we have
\begin{align*}
  \limsup_{n\to \infty} \|e^{it\Delta_{\mathbb{R}^2}} \vec{w}_n^J\|_{L_{t,x}^4 h^{1-\epsilon_0}}
\lesssim &  \limsup_{n\to \infty} \|e^{it\Delta_{\mathbb{R}^2}} \vec{w}_n^J\|_{L_{t,x}^4 h^{1 }}^{1-\epsilon_0} \|e^{it\Delta_{\mathbb{R}^2}} \vec{w}_n^J\|_{L_{t,x}^4 l^2}^{\epsilon_0} \\
\lesssim  & \limsup_{n\to \infty} \|\vec{w}_n^J\|_{L_x^2 h^1}^{1-\epsilon_0} \|e^{it\Delta_{\mathbb{R}^2}} \vec{w}_n^J\|_{L_{t,x}^4 l^{\frac{3}{2}}} ^{\frac{3 \epsilon_0}{5}}  \|e^{it\Delta_{\mathbb{R}^2}} \vec{w}_n^J\|_{L_{t,x}^4 l^4}^{\frac{2\epsilon_0}{5}} \\
\lesssim & \limsup_{n\to \infty} \|\vec{w}_n^J\|_{L_x^2 h^1}^{1-\epsilon_0} \|e^{it\Delta_{\mathbb{R}^2}} \vec{w}_n^J\|_{L_{t,x}^4 h^1} ^{\frac{3 \epsilon_0}{5}}  \|e^{it\Delta_{\mathbb{R}^2}} \vec{w}_n^J\|_{L_{t,x}^4 l^4}^{\frac{2\epsilon_0}{5}}\\
\lesssim & \limsup_{n\to \infty} \|\vec{w}_n^J\|_{L_x^2 h^1}^{1-\frac25 \epsilon_0}\|e^{it\Delta_{\mathbb{R}^2}} \vec{w}_n^J\|_{L_{t,x}^4 l^4}^{\frac{2\epsilon_0}{5}}
  \to 0, \text{ as }  J\to \infty.
\end{align*}

\end{remark}
To prove the scattering, recalling Section 3, it suffices to prove that, for $\vec{u}$ solving \eqref{maineq} with initial data $\vec{u}_0$,
$$\|\vec{u}\|_{L_{t,x}^{4}l^2(\mathbb{R}\times\mathbb{R}^{2}\times\mathbb{Z}^2)}<\infty,$$ for all $\vec{u}_{0}\in L^2h^1(\mathbb{R}^2\times\mathbb{Z}^2).$

Then for $\vec{u}$ solving \eqref{maineq} with maximal lifespan interval $I$, we define the function
$$A(m)=\sup\{\|\vec{u}\|_{L_{t,x}^{4}l^2(I\times\mathbb{R}^{2}\times\mathbb{Z}^2)}: \|\vec{u}(0)\|^2_{L_x^2 h^1(\mathbb{R}^2 \times\mathbb{Z}^2)}\leq m\},$$ and $$m_{0}=\sup\{m: A(m')<+\infty, \forall m'<m\}.$$
If we can prove $m_{0}=+\infty$, then global well-posedness and scattering are established.

With Proposition \ref{pr6.640}, the definition of the critical mass $m_0$ above, following the argument in section 5 of \cite{TVZ1}, we have the following Palais-Smale condition modulo G property, which asserts a certain compactness modulo G in blowup sequences of solutions with mass less than or equal to the critical mass. It connects the low-level scattering norm $L_{t,x}^{4}l^2(I\times\mathbb{R}^{2}\times\mathbb{Z}^2)$  and the compactness characterization of high-level space $H^{1}h^{1}(\mathbb{R}^2\times\mathbb{Z}^2)$ at discrete direction of ``j" together. We will give a sketch of the proof below.
\begin{proposition}\label{pr-y4.4}
Let $u_{n,p}(t,x)$ be defined on time interval $I_{n}.$
Assume $m_{0}<+\infty$,  $\vec{u}_{n}:=\left\{u_{n,p}\right\}_{p\in\mathbb{Z}^2}, (n=1,2,\cdot\cdot\cdot)$ is a sequence of solutions to \eqref{maineq} satisfying $$\limsup\limits_{n\rightarrow\infty}M\left(\vec{u}_{n}\right)=m_{0},$$ $$\lim\limits_{n\rightarrow\infty}\left\|\vec{u}_{n}\right\|_{L_{t,x}^{4}l^2(I_{n}\cap(-\infty,t_{n}))}=
\lim\limits_{n\rightarrow\infty}\left\|\vec{u}_{n}\right\|_{L_{t,x}^{4}l^2(I_{n}\cap(t_{n},+\infty))}=+\infty, \,\,for\,\, some \,\,t_{n}\in I_{n}.$$
Then $G\vec{u}_{n}(t_{n})$ converges (up to subsequence) in $G\backslash H^{1}h^{1}(\mathbb{R}^2\times\mathbb{Z}^2).$
\end{proposition}
\begin{proof}[Sketch of the proof] By translating $\vec{u}_{n}$ (and $I_n$) in time, we may take
$t_n = 0$ for all $n$, thus,
$$\lim\limits_{n\rightarrow\infty}\left\|\vec{u}_{n}\right\|_{L_{t,x}^{4}l^2(I_n\cap(-\infty,0))}=
\lim\limits_{n\rightarrow\infty}\left\|\vec{u}_{n}\right\|_{L_{t,x}^{4}l^2(I_n\cap(0,+\infty))}=+\infty.$$
$\limsup\limits_{n\rightarrow\infty}M\left(\vec{u}_{n}\right)=m_{0}$ implies that $\{\vec{u}_{n}\}$ is bounded  (passing to a subsequence if necessary) in $L_x^2h^1$.  Therefore, by Proposition \ref{pr6.640}, we have
$$\vec{u}_n(0,x) =   \sum_{j=1}^J g_n^j e^{it_n^j \Delta_{\mathbb{R}^2}} \vec{\phi}^j + \vec{w}_n^J(x),$$
where $t_n^{j}\in \mathbb{R}$,  $g_n^{j}\in G$. By extracting subsequence and time translation, we can assume $t_n^{j}\rightarrow t^{j}\in \{-\infty,0,+\infty\}$ as $n\rightarrow \infty$.

We now define a nonlinear profile $\vec{v}^{j}: \mathbb{R} \times (\mathbb{R}^{2}\times\mathbb{Z}^{2}) \rightarrow \mathbb{C}$ associated to $\vec{\phi}^{j}$ and depending on the limiting value of $t_{n}^{j}$, as follows:
\begin{enumerate}
	\item If $t_{n}^{j}$ is identically zero, we define $\vec{v}^{j}$ to be the maximal-lifespan solution with initial data $\vec{v}^{j}(0)=\vec{\phi}^{j}$.
	\item If $t_{n}^{j}$ converges to $+\infty$, we define $\vec{v}^{j}$ to be the maximal-lifespan solution which scatters forward in time to $e^{i t \Delta} \vec{\phi}^{j}$.
	\item If $t_{n}^{j}$ converges to $-\infty$, we define $\vec{v}^{j}$ to be the maximal-lifespan solution which scatters backward in time to $e^{i t \Delta} \vec{\phi}^{j}$.
\end{enumerate}
Similarly to the proof in \cite{CGZ,HP} to deal with the quintic nonlinear Schr\"odinger equation on $\mathbb{R} \times \mathbb{T}^{2}$, we can obtain the following decoupling property for the nonlinear profiles defined above. We also refer to \cite{MV1} for the original argument.
\begin{lemma}[\bf Decoupling of nonlinear profiles]\label{le-y4.3}
Let $\vec v_{n}^{j}:=T_{g_{n}^{(j)}}\left[\vec v^{j}\left(\cdot+t_{n}^{j}\right)\right]$ be the nonlinear solutions defined above. Then, for $j \neq k$, $0<\theta<1$,
  $$\left\|\|\vec{v}^{j}_{n}\|_{h^{\beta}} \|\vec{v}^{k}_{n}\|_{l^{2}}\right\|_{L_{t, x}^{2}} \rightarrow 0$$
   $$\left\|\|\vec{v}^{j}_{n}\|^{1-\theta}_{h^{\beta}} \|\vec{v}^{k}_{n}\|^{\theta}_{l^{2}}\right\|_{L_{t, x}^{4}} \rightarrow 0$$as $n \rightarrow \infty$.
\end{lemma}
We then define the approximant $\vec{u}_{n}^{(l)} \in C_{t}^{0} L_{x}^{2}h^1\left(\mathbb{R} \times \mathbb{R}^{2}\times \mathbb{Z}^{2}\right)$ to $\vec{u}_{n}$ for $n, l=1,2, \ldots$ by the formula
\begin{equation}
\vec{u}_{n}^{l}(t):=\sum_{j=1}^{l} T_{g_{n}^{(j)}}\left[\vec{v}^{j}\left(\cdot+t_{n}^{j}\right)\right](t)+e^{i t \Delta} \vec{w}_{n}^{l}.	
\end{equation}
\begin{lemma}[Asymptotic solvability of equation]
We have
$$
\lim _{l \rightarrow \infty} \limsup _{n \rightarrow \infty}\left\|\left(i \partial_{t}+\Delta\right) \vec{u}_{n}^{l}-\vec{F}\left(\vec{u}_{n}^{l}\right)\right\|_{L_{t, x}^{4 /3}l^2\left(\mathbb{R} \times \mathbb{R}^{2}\times \mathbb{Z}^{2}\right)}=0.
$$
\end{lemma}
\proof Write
$$
\vec v_{n}^{j}:=T_{g_{n}^{(j)}}\left[\vec v^{j}\left(\cdot+t_{n}^{j}\right)\right].
$$
By the definition of $\vec u_{n}^{l}$,  so it suffices by the triangle inequality to show that
$$
\lim _{l \rightarrow \infty} \limsup _{n \rightarrow \infty}\left\|\vec F\left(\vec u_{n}^{l}-e^{i t \Delta} \vec w_{n}^{l}\right)-\vec F\left(\vec u_{n}^{l}\right)\right\|_{L_{t, x}^{4 /3}l^2\left(\mathbb{R} \times \mathbb{R}^{2}\times \mathbb{Z}^{2}\right)}=0
$$
and
$$
\lim _{n \rightarrow \infty}\left\|\vec F\left(\sum_{j=1}^{l} \vec v_{n}^{j}\right)-\sum_{j=1}^{l} \vec F\left(\vec v_{n}^{j}\right)\right\|_{L_{t, x}^{4/3}l^2\left(\mathbb{R} \times \mathbb{R}^{2}\times \mathbb{Z}^{2}\right)}=0
$$
for each $l$.
The first inequality follows immediately from Lemma \ref{es:resonant} and Remark \ref{re-y4.2}. For the second inequality, we use the elementary inequality
$$
\left|F\left(\sum_{j=1}^{l} z_{j}\right)-\sum_{j=1}^{l} F\left(z_{j}\right)\right| \leq C_{l, d} \sum_{j \neq j^{\prime}}\left|z_{j}\right|\left|z_{j}^{\prime}\right|^{2}
$$
for some $C_{l, d}<\infty$, and the claim follows from Lemma \ref{es:resonant} and lemma \ref{le-y4.3}.

Then we combine this lemma and $\lim _{n \rightarrow \infty} M\left(\vec u_{n}^{l}(0)-\vec u_{n}(0)\right)=0$, using item 3 of Theorem \ref{mainloc}, continuing to follow the argument in Section 5 of \cite{TVZ1}, we can show $\sup\limits_{j}M(\vec{\phi}^{j})=m_0$. where $\vec{\phi}^{j}=\{\phi_p^{j}\}_{p\in \mathbb{Z}}$, this implies $J=1$, $M(\vec{\phi}^{1})=m_0$ and
$$\vec{u}_{n}(0,x)= g_n^{1} e^{it_n^{1} \triangle} \vec{\phi}^{1} +\vec{w}_n^{1}=:g_{n} e^{it_{n} \triangle} \vec{\phi} +\vec{w}_{n}.$$
Furthermore, we obtain $\lim_{n\rightarrow\infty}t_{n}=0$. This shows Proposition \ref{pr-y4.4} is true.
	
\end{proof}
 Therefore, similar to Theorem 3.3 in \cite{YZ}, we have
\begin{theorem}[Reduction to almost periodic solutions]\label{th-y4.3}
Assume $m_{0}<+\infty$. Then there exits a solution (calling critical element) $\vec{u}\in C_{t}^{0}L_{x}^{2}h^{1}\left(I\times \mathbb{R}^2\times \mathbb{Z}^2\right)\bigcap L_{t}^{4}L_{x}^{4}l^2\left(I\times \mathbb{R}^2\times \mathbb{Z}^2\right)$ to \eqref{maineq} with $I$ the maximal lifespan interval such that
\begin{enumerate}
    \item $M(\vec{u})=m_{0},$
    \item $\vec{u}$ blows up at both directions in time, i.e. $\|\vec{u}\|_{L_{t,x}^{4}l^2(I\cap(-\infty,t_{0}))}=
        \|\vec{u}\|_{L_{t,x}^{4}l^2(I\cap(\widetilde{t}_{0},+\infty))}=+\infty$, for some $t_{0},\widetilde{t}_{0}\in I,$
    \item $\vec{u}$ is an almost periodic solution modulo $G,$
\end{enumerate}
where $\vec{u}$ is called an almost periodic solution modulo $G$ if the quotiented orbit $\{G\vec{u}:t\in I\}$ is a precompact subset of $G\backslash H^{1}h^{1}(\mathbb{R}^2\times\mathbb{Z}^2).$
\end{theorem}

 Finally, we phrase the property of almost periodicity modulo $G$ of the solution to \eqref{eq-main1.1} in a ``quantitative'' version.
\begin{proposition}\label{pr-y4.8}
The following statements are equivalent.
\begin{enumerate}
  \item $\vec{u}\in C_{t,loc}^{0}L^2h^1(I\times\mathbb{R}^2\times\mathbb{Z}^2)$ is almost periodic modulo $G.$
  \item $\{G\vec{u}(t):t\in I\}$ is precompact in $G\backslash H^{1}h^{1}(\mathbb{R}^2\times\mathbb{Z}^2).$
  \item $\exists x(t),\xi(t),N(t)$ such that  $\forall\eta>0$, $\exists K(\eta)>0,R(\eta)>0$ such that  for $t\in I,$
\begin{align}\label{eq-y4.2}
  \sum\limits_{|j|>K(\eta)}\langle j\rangle^{2}\|u_{j}\|_{L^2}^{2}&<\eta,\\ \label{eq-y4.3}
  \sum\limits_{|j|=0}^{K(\eta)} \langle j\rangle^{2}\int_{|x-x(t)|\geq \frac{R(\eta)}{N(t)}}|u_{j}(x)|^{2}\mathrm{d}x&<\eta, \\ \label{eq-y4.4}
  \sum\limits_{|j|=0}^{K(\eta)} \langle j\rangle^{2}\int_{|\xi-\xi(t)|\geq R(\eta)N(t)}|\hat{u}_{j}(\xi)|^{2}\mathrm{d}\xi&<\eta.
\end{align}
\end{enumerate}
\end{proposition}
\begin{corollary}\label{co-y4.9}
$\vec{u}\in C_{t,loc}^{0}L^2h^1(I\times\mathbb{R}^2\times\mathbb{Z}^2)$ is almost periodic modulo $G$. Then there exist $x(t),\xi(t),N(t)$ such that for arbitrary $\eta>0$, there exists $ R(\eta)>0$ such that for $t\in I,$
\begin{equation}\label{eq-y3.6}
\sum\limits_{j\in\mathbb{Z}^2} \langle j\rangle^{2}\left[\int_{|x-x(t)|\geq \frac{R(\eta)}{N(t)}}|u_{j}(x)|^{2}\mathrm{d}x+\int_{|\xi-\xi(t)|\geq R(\eta)N(t)}|\hat{u}_{j}(\xi)|^{2}\mathrm{d}\xi\right]<\eta.
\end{equation}
\end{corollary}

\begin{theorem}[The estimate of $N(t)$]\label{th-y4.5}The following statements hold:
\begin{enumerate}
  \item For any nonzero almost periodic solution $\vec{u}$ to \eqref{eq-main1.1} there exists $\delta(\vec{u}) > 0$ such that for any $t_0 \in I$,
      $$\|\vec{u}\|_{L^4_{t,x}h^{1}([t_0,t_0+\frac{\delta}{N(t_0)^2}]\times\mathbb{R}^2)}\sim \|\vec{u}\|_{L^4_{t,x}h^{1}([t_0-\frac{\delta}{N(t_0)^2},t_0]\times\mathbb{R}^2)}\sim1.$$
  \item If $J$ is an interval with $\|\vec{u}\|_{L^4_{t,x}h^{1}(J\times\mathbb{R}^2)}=1$, then for $t_1,t_2\in J$,  $N(t_1)\sim_{m_0} N(t_2)$, and $|\xi(t_1)-\xi(t_2)|\lesssim N(J_k)$, where $N(J_k):=\sup_{t\in J_k}N(t)$.
  \item  Suppose $\vec{u}$ is a minimal mass blowup solution with $N(t) \leq 1$.  Suppose also that $J$
   is some interval partitioned into subintervals $J_k$ with $\|\vec{u}\|_{L^4_{t,x}h^{1}(J_k\times\mathbb{R}^2)}= 1$ on each $J_k$, then $N(J_k)\sim \int_{J_k}N(t)^3 dt\sim \inf_{t\in J_k}N(t)$ and $\sum_{J_k}N(J_k)\sim \int_{J}N(t)^3 dt$.
   \item  If $\vec{u}(t,x)$ is a minimal mass blowup solution on an interval $J$, then
   $$\int_J N(t)^2dt\lesssim \|\vec{u}\|^{4}_{L_{t,x}^{4}h^{1}(J\times \mathbb{R}^2)}\lesssim 1+\int_J N(t)^2dt.$$
\end{enumerate}
\end{theorem}
\begin{proof}
  The proof is similar to that of Lemma 2.12, Lemma 2.13 and Lemma 2.15 in \cite{D} and we omit it.
\end{proof}
\begin{remark}\label{re-y4.10}
By Theorem \ref{th-y4.5}, $|N'(t)|,|\xi'(t)|\lesssim N(t)^3$. We can use this fact to control the movement of $\xi(t)$.
\end{remark}

Fix three constants $0<\epsilon_3 \ll \epsilon_2\ll \epsilon_1<1$ in the following. By Corollary \ref{co-y4.9}  and Remark \ref{re-y4.10}, $\epsilon_1,\epsilon_2,\epsilon_3$ can also satisfy
\begin{equation}\label{eq-y5.1}
|N'(t)|+|\xi'(t)|\leq2^{-20}\frac{N(t)^3}{\epsilon_1^{1/2}},
\end{equation}
\begin{equation}\label{eq-y5.2}
\sum\limits_{j\in\mathbb{Z}} \langle j\rangle^{2}\left[\int_{|x-x(t)|\geq \frac{2^{-20}\epsilon_3^{-1/4}}{N(t)}}|u_{j}(x)|^{2}\mathrm{d}x+\int_{|\xi-\xi(t)|\geq 2^{-20}\epsilon_3^{-1/4}N(t)}|\hat{u}_{j}(\xi)|^{2}\mathrm{d}\xi\right]<\epsilon_2^2,
\end{equation}
and
$$\epsilon_3<\epsilon_2^{10}.$$
Suppose $M=2^{k_0}$ is a dyadic integer with $k_0\geq0$. Let $[0,T]$ be an interval such that $\|\vec{u}\|^4_{L^4_{t,x}h^{1}([0,T])}=M$ and $\int_0^T N(t)^3 dt=\epsilon_3M$. Partition $[0,T]=\cup_{l=0}^{M-1} J_l$ with $\|\vec{u}\|_{L^4_{t,x}h^{1}(J_l)}=1$, we call the intervals $J_l$'s small intervals.

\begin{definition}\label{de-y5.1}
For an integer $0\leq j<k_0$, $0\leq k<2^{k_0-j}$, let
$$G_{k}^{j}=\cup_{\alpha=k2^j}^{(k+1)2^j-1}J^{\alpha}.$$
Where $J^{\alpha}$'s satisfy $[0,T]=\cup_{\alpha=0}^{M-1} J^{\alpha}$  with
\begin{equation}\label{eq-y5.33}
\int_{J^{\alpha}} \big( N(t)^3 + \epsilon_3\|\vec{u}(t)\|^4_{L_x^4h^{1}(\mathbb{R}^2\times \mathbb{Z})} \big)dt=2\epsilon_3.
\end{equation}
 For$j\geq k_0$ let $G_{k}^j=[0,T].$
 Now suppose that $G_k^j=[t_0,t_1]$, let $\xi(G_k^j)=\xi(t_0)$ and define $\xi(J_l)$, $\xi(J^{\alpha})$ in a similar manner.
\end{definition}
We collect some useful properties of the small intervals below,
\begin{proposition}\label{re-y5.1}
\begin{enumerate}
\item It follows from Theorem \ref{th-y4.5} that $N(J_l)\sim \int_{J_l}N(t)^3 dt\sim \inf_{t\in J_l}N(t)$. Additionally, by \eqref{eq-y5.33}, we obtain
   \begin{equation}\label{eq-y5.34}
   \sum_{J_l\subset G_k^j}N(J_l)\lesssim \sum_{J_l\subset G_k^j}\int_{J_l}N(t)^3dt\lesssim \int_{G_k^j}N(t)^3dt\lesssim \sum_{\alpha=k2^j}^{(k+1)2^j-1}\int_{J^{\alpha}}N(t)^3dt\lesssim 2^j\epsilon_3.
   \end{equation}
\item By \eqref{eq-y5.1} and Definition \ref{de-y5.1},  for all $t \in G_k^{j}$,
  \begin{equation}
  |\xi(t)-\xi(G_k^{j})|\leq \int_{G_k^{j}}2^{-20}\epsilon_{1}^{-1/2}N(t)^3dt \leq 2^{j-19}\epsilon_3 \epsilon_{1}^{-1/2}.
  \end{equation}
  Therefore, for all $t \in G_k^{j}$ and $i\geq j$,
  $$\{\xi:2^{i-1}\leq |\xi-\xi(t)|\leq 2^{i+1} \} \subset \{\xi:2^{i-2}\leq |\xi-\xi(G_k^{j})|\leq 2^{i+2} \} \subset \{\xi:2^{i-3}\leq |\xi-\xi(t)|\leq 2^{i+3} \},$$
  and
  $$\{\xi:|\xi-\xi(t)|\leq 2^{i+1} \} \subset \{\xi: |\xi-\xi(G_k^{j})|\leq 2^{i+2} \} \subset \{\xi: |\xi-\xi(t)|\leq 2^{i+3} \}.$$

 \item Suppose $\vec{u}(t)$ is a minimal mass blowup solution to \eqref{eq-y1}. If $J$ is a time interval with
$\|\vec{u}\|_{L^4_{t,x}h^{1}(J)}\lesssim 1$, then
$$\|\vec{u}\|_{U^2_{\triangle}(h^{1};J)}\lesssim1 \ \ and \ \ \left\|P_{\xi(t_0),>\frac{N(J)}{2^4\epsilon_3^{1/4}}} \vec{u}\right\|_{U^2_{\triangle}(h^{1};J)}\lesssim \epsilon_2,$$ for any $t_0\in J$, where $N(J)=\sup_{t\in J}N(t)$. Furthermore, by Proposition \ref{pr-y5.3}(6),
 $$\|\vec{u}\|_{L^p_tL^q_x h^{1}(J)}\lesssim 1,\,\,\, \|P_{\xi(t_0),>\frac{N(J)}{2^4\epsilon_3^{1/4}}} \vec{u}\|_{L^p_tL^q_x h^{1}(J)}\lesssim \epsilon_2,$$
for $(p,q)$ admissible pair and $p>2$.

\item If $N(J)< 2^{i-5}\epsilon_3^{1/2},$ then
$$\|P_{\xi(G_{\alpha}^{i}), i-2 \leq \cdot \leq i+2} \mathbf{F}(\vec{u})\|_{L_t^{3/2}L_x^{6/5}h^{1}(J)}\lesssim \big\|P_{\xi(J),>\frac{N(J)}{2^{4}\epsilon_3^{1/4}}} \vec{u}\big\|_{L_t^{\infty}L^2_xl^2(J)} \|\vec{u}\|^2_{L_t^3L_x^6h^{1}(J)}\lesssim \epsilon_2.$$
So for $0\leq i \leq 11$, $N(G_{\alpha}^{i})< 2^{i-5}\epsilon_3^{1/2},$ since $G_{\alpha}^{i}$ is a union of $\leq 2^{11}$ such small intervals,
$$\|P_{\xi(G_{\alpha}^{i}), i-2 \leq \cdot \leq i+2} \mathbf{F}(\vec{u})\|_{L_t^{3/2}L_x^{6/5}h^{1}(G_{\alpha}^{i})}\lesssim \epsilon_2.$$
\end{enumerate}
\end{proposition}

Now we define our spaces as in the Section 3 of \cite{D}, in which we derive the long time Strichartz estimates.
\begin{definition}[$\tilde{X}_{k_0}$ spaces]\label{de-y5.2}
For any $G_k^{j}\subset[0,T]$ let
\begin{equation}\label{eq-y5.3'}
  \|\vec{u}\|^2_{X(G_k^{j})}:=\sum_{0\leq i< j}2^{i-j}\sum_{G_{\alpha}^{i}\subset G_k^{j}}\|P_{\xi(G_{\alpha}^{i}), i-2\leq\cdot\leq i+2}\vec{u}\|^2_{U^2_{\triangle}(h^{1};G_{\alpha}^{i}\times\mathbb{R}^2)}
  +\sum_{i\geq j}\|P_{\xi(G_{k}^{j}),i-2\leq\cdot\leq i+2} \vec{u}\|^2_{U^2_{\triangle}(h^{1};G_{k}^{j}\times\mathbb{R}^2)}.
\end{equation}
Here $P_{\xi(t),i-2\leq\cdot\leq i+2} \vec{u}=e^{ix \cdot \xi(t)}P_{i-2\leq\cdot\leq i+2}(e^{-ix\cdot \xi(t)} \vec{u})$ with $P_{i-2\leq\cdot\leq i+2}$ being the Littlewood-Paley projector.\\
Then  define $\tilde{X}_{k_0}$ to be the supremum of \eqref{eq-y5.3'} over all intervals $G_k^j\subset[0,T]$ with $k\leq k_0$.
\begin{equation}\label{X}
\|\vec{u}\|^2_{\tilde{X}_{k_0}([0,T])}:=\sup_{0\leq j\leq k_0}\sup_{G_k^{j}\subset[0,T]}\|\vec{u}\|^2_{X(G_k^{j})}.
\end{equation}
Also for $0\leq k_{\ast}\leq k_0$, let
\begin{equation}\label{XXX}
\|\vec{u}\|^2_{\tilde{X}_{k_{\ast}}([0,T])}:=\sup_{0\leq j\leq k_{\ast}}\sup_{G_k^{j}\subset[0,T]}\|\vec{u}\|^2_{X(G_k^{j})}.
\end{equation}
\end{definition}

\begin{definition}[$\tilde{Y}_{k_0}$ spaces]
The $\tilde{Y}_{k_0}$ norm measures the $\tilde{X}_{k_0}$ norm of $\vec{u}$ at scales much higher than $N(t)$. This norm provides some crucial ``smallness", closing a bootstrap argument in the next section. Let
\begin{equation}
\begin{split}
  \|\vec{u}\|^2_{Y(G_k^{j})}&:=\sum_{0< i< j}2^{i-j}\sum_{G_{\alpha}^{i}\subset G_k^{j}:N(G_{\alpha}^{i})\leq 2^{i-5}\epsilon_3^{1/2}}\|P_{\xi(G_{\alpha}^{i}), i-2\leq\cdot\leq i+2} \vec{u}\|^2_{U^2_{\triangle}(h^{1};G_{\alpha}^{i}\times\mathbb{R}^2)} \\
    & +\sum_{i\geq j,i>0:N(G_{k}^{j})\leq 2^{i-5}\epsilon_3^{1/2}}\|P_{\xi(G_{k}^{j}),i-2\leq\cdot\leq i+2} \vec{u}\|^2_{U^2_{\triangle}(h^{1};G_{k}^{j}\times\mathbb{R}^2)}.
\end{split}
\end{equation}
Define $\|\vec{u}\|_{\tilde{Y}_{k_{\ast}}([0,T])}$
using $\|\vec{u}\|_{Y(G_k^{j})}$ in the same way as $\|\vec{u}\|_{\tilde{X}_{k_{\ast}}([0,T])}$ was done.
\end{definition}

After giving the long time Strichartz norms, we should point out the relationship between $L^p_tL^q_xh^{1}$ norm and the long time Strichartz norms, which can be easily obtained from the definition of the long time Strichartz norms as to Lemma 5.7 in  \cite{YZ}.
\begin{lemma}\label{le-y5.7}
For $i < j$, $(p,q)$ an admissible pair, we have
\begin{equation}\label{eq-y5.9}
  \|P_{\xi(t),i}\vec{u}\|_{L^p_tL^q_xh^1(G^j_k\times\mathbb{R}^2)}\lesssim 2^{\frac{j-i}{p}} \|\vec{u}\|_{\tilde{X}_j(G^j_k)}.
\end{equation}
\begin{equation}\label{eq-y5.9'}
  \|P_{\xi(t),\geq j}\vec{u}\|_{L^p_tL^q_xh^1(G^j_k\times\mathbb{R}^2)}\lesssim  \|\vec{u}\|_{X(G^j_k)}.
\end{equation}
\end{lemma}

\section{Long time Strichartz estimate and Frequency localized interaction Morawetz estimate}
In this section, we discuss long time Strichartz estimate and frequency localized interaction Morawetz estimate. These methods are developed in \cite{D} and are essential to prove Theorem \ref{mainthm}.

In contrast to the long time Strichartz estimate established  in Section 5 in \cite{YZ}, we want to establish the long time Strichartz estimate on the``$h^1$" level, instead of on the ``$l^2$" level, at the discrete direction of ``$j$". This is one of the main technical differences as explained in Section 1. We follow the argument which is used to solve the scattering theory for 2d mass-critical NLS by Dodson in the Section 5 of \cite{D}. The similar nonlinear estimate $\|\vec{F}(\vec{u})\|_{h^1}\lesssim \|\vec{u}\|^3_{h^1}$ on the ``$h^1$" level at the discrete direction of ``$j$" guarantees that  just by  changing  $l^2$ norm in Section 6 of \cite{YZ} into $h^1$ norm at the discrete direction of ``$j$",  we can get the three corresponding  bilinear estimates at the cost of some loss of symmetry, while we can pay for the cost by  those cancellations in last section. the distinctive terms which are caused when we change  $l^2$ norm in Section of \cite{YZ} into $h^1$ norm at the discrete direction of ``$j$" appear in the estimates of $F_{2,j'}$  (6.33) in Section of \cite{YZ}, In the following we will pick out and deal with the distinctive terms in the estimates of $F_{2,j'}$ when we change  $l^2$ norm in Section 6 of \cite{YZ} into $h^1$ norm at the discrete direction of ``$j$".
\subsection{Three bilinear Strichartz estimates and Long time Strichartz estimate.}
\begin{equation}\label{eq-y5.90}
\begin{split}
F_{2,j'}(t,y)=& \Im\big[ 2\bar{u}^l_{j'} P_{\xi(t),\leq l_2} \big(\sum_{\mathcal{R}(j')}u^h_{j'_1}\bar{u}^h_{j'_2}u^l_{j'_3}\big) + 2(P_{\xi(t),\leq l_2}\bar{u}^h_{j'}) P_{\xi(t),\leq l_2} \big(\sum_{\mathcal{R}(j')}u^h_{j'_1}\bar{u}^l_{j'_2}u^l_{j'_3}\big)\\
&+ \bar{u}^l_{j'} P_{\xi(t),\leq l_2} \big(\sum_{\mathcal{R}(j')}u^h_{j'_1}\bar{u}^l_{j'_2}u^h_{j'_3}\big) + (P_{\xi(t),\leq l_2}\bar{u}^h_{j'}) P_{\xi(t),\leq l_2} \big(\sum_{\mathcal{R}(j')}u^l_{j'_1}\bar{u}^h_{j'_2}u^l_{j'_3}\big)\big]\\
&=:\uppercase\expandafter{\romannumeral1}+\uppercase\expandafter{\romannumeral2},
\end{split}
\end{equation}
where
$$\uppercase\expandafter{\romannumeral1}:=\Im\big[ 2\bar{u}^l_{j'} P_{\xi(t),\leq l_2} \big(\sum_{\mathcal{R}(j')}u^h_{j'_1}\bar{u}^h_{j'_2}u^l_{j'_3}\big) + 2(P_{\xi(t),\leq l_2}\bar{u}^h_{j'}) P_{\xi(t),\leq l_2} \big(\sum_{\mathcal{R}(j')}u^h_{j'_1}\bar{u}^l_{j'_2}u^l_{j'_3}\big)\big]$$
and
$$\uppercase\expandafter{\romannumeral2}:= \Im\big[\bar{u}^l_{j'} P_{\xi(t),\leq l_2} \big(\sum_{\mathcal{R}(j')}u^h_{j'_1}\bar{u}^l_{j'_2}u^h_{j'_3}\big) + (P_{\xi(t),\leq l_2}\bar{u}^h_{j'}) P_{\xi(t),\leq l_2} \big(\sum_{\mathcal{R}(j')}u^l_{j'_1}\bar{u}^h_{j'_2}u^l_{j'_3}\big)\big].$$
We can estimate ``$\uppercase\expandafter{\romannumeral2}$" as to (6.38) in \cite{YZ}  when we change  $l^2$ norm  into $h^1$ norm at the discrete direction of ``$j$". So it remains ``$\uppercase\expandafter{\romannumeral1}$" for us to deal with. Observe that
\begin{equation}
\begin{split}
  & \Im\Big(\sum\limits_{\left(j, j_{1}, j_{2}, j_{3}\right) \in \mathcal{R}}\langle j\rangle^{2}[ (P_{\xi(t),\leq l_2}\bar{u}^h_{j})  u^l_{j_{1}} \bar{u}^l_{j_{2}} (P_{\xi(t),\leq l_2}u^h_{j_{3}})+\bar{u}^l_{j}  (P_{\xi(t),\leq l_2}u^h_{j_{1}}) (P_{\xi(t),\leq l_2}\bar{u}^h_{j_{2}}) u^l_{j_{3}}]\Big) \\
  =&\Im\Big(\sum\limits_{\left(j, j_{1}, j_{2}, j_{3}\right) \in \mathcal{R}}(\langle j\rangle^{2}+\langle j_2\rangle^{2})[ (P_{\xi(t),\leq l_2}\bar{u}^h_{j})  u^l_{j_{1}} \bar{u}^l_{j_{2}} (P_{\xi(t),\leq l_2}u^h_{j_{3}})]\Big)\\
  =&0. \\
\end{split}
\end{equation}

Therefore,
\begin{equation}\label{eq-ykl5.91}
\begin{split}
&2\Im\big[\sum_{j'} \langle j'\rangle^{2}  \bar{u}^l_{j'} P_{\xi(t),\leq l_2} \big(\sum_{\mathcal{R}(j')}u^h_{j'_1}\bar{u}^h_{j'_2}u^l_{j'_3}\big) + \sum_{j'} \langle j'\rangle^{2}  (P_{\xi(t),\leq l_2}\bar{u}^h_{j'}) P_{\xi(t),\leq l_2} \big(\sum_{\mathcal{R}(j')}u^h_{j'_1}\bar{u}^l_{j'_2}u^l_{j'_3}\big)\big]\\
=&2\Im\big[\sum_{\left(j', j'_{1}, j'_{2}, j'_{3}\right) \in \mathcal{R}} \langle j'\rangle^{2} \big(\bar{u}^l_{j'}P_{\xi(t),\leq l_2}(u^h_{j'_1}\bar{u}^h_{j'_2}u^l_{j'_3}) -\bar{u}^l_{j'} u^l_{j'_3} ( P_{\xi(t),\leq l_2}u^h_{j'_1})(P_{\xi(t),\leq l_2}\bar{u}^h_{j'_2})\big)\big]\\
&+2\Im\big[\sum_{\left(j', j'_{1}, j'_{2}, j'_{3}\right) \in \mathcal{R}} \langle j'\rangle^{2} \big((P_{\xi(t),\leq l_2}\bar{u}^h_{j'})P_{\xi(t),\leq l_2} (u^h_{j'_1}\bar{u}^l_{j'_2}u^l_{j'_3}) - (P_{\xi(t),\leq l_2}\bar{u}^h_{j'})  u^l_{j'_1} \bar{u}^l_{j'_2} (P_{\xi(t),\leq l_2}u^h_{j'_3}) \big)\big]\\
=:&\eqref{eq-ykl5.91}(a)+\eqref{eq-ykl5.91}(b).
\end{split}
\end{equation}
\eqref{eq-ykl5.91}(b) can be easily estimated just like (6.34)(b) in \cite{YZ}. So we only need to estimate \eqref{eq-ykl5.91}(a) hereinafter.
\begin{equation}\label{eq-ykl5.91'}
\begin{split}
\eqref{eq-ykl5.91}(a):=&2\Im\big[\sum_{\left(j', j'_{1}, j'_{2}, j'_{3}\right) \in \mathcal{R}} \langle j'\rangle^{2} \big(\bar{u}^l_{j'}P_{\xi(t),\leq l_2}(u^h_{j'_1}\bar{u}^h_{j'_2}u^l_{j'_3}) -\bar{u}^l_{j'} u^l_{j'_3} ( P_{\xi(t),\leq l_2}u^h_{j'_1})(P_{\xi(t),\leq l_2}\bar{u}^h_{j'_2})\big)\big]\\
=&2\Im\big[\sum_{\left(j', j'_{1}, j'_{2}, j'_{3}\right) \in \mathcal{R}} \langle j'\rangle^{2} \big(\bar{u}^l_{j'}P_{\xi(t),\leq l_2}(u^h_{j'_1}\bar{u}^h_{j'_2}u^l_{j'_3}) - \bar{u}^l_{j'} u^l_{j'_3} ( P_{\leq l_2}(u^h_{j'_1}\bar{u}^h_{j'_2}))\big)\big]\\
&+2\Im\big[\sum_{\left(j', j'_{1}, j'_{2}, j'_{3}\right) \in \mathcal{R}} \langle j'\rangle^{2} \big(\bar{u}^l_{j'} u^l_{j'_3} ( P_{\leq l_2}(u^h_{j'_1}\bar{u}^h_{j'_2}))-
\bar{u}^l_{j'} u^l_{j'_3} ( P_{\xi(t),\leq l_2}u^h_{j'_1})(P_{\xi(t),\leq l_2}\bar{u}^h_{j'_2})\big)\big]\\
=:&\eqref{eq-ykl5.91'}(a)'+\eqref{eq-ykl5.91'}(c).
\end{split}
\end{equation}
The only difference between \eqref{eq-ykl5.91'}$(a)'$ and (6.34)(a) in \cite{YZ} is that the summation of ``$j'$" goes from ``$l^2$" level to the ``$h^1$" level. So by using the nonlinear estimate $\|\vec{F}(\vec{u})\|_{h^1}\lesssim \|\vec{u}\|^3_{h^1}$ at the discrete direction of ``$j'$" and repeating the process of estimating of (6.34)(a) in \cite{YZ}, we can similarly estimate  \eqref{eq-ykl5.91'}$(a)'$. Next we estimate \eqref{eq-ykl5.91'}(c).

\begin{equation}\label{eq-ykl5.92'}
\begin{split}
&\eqref{eq-ykl5.91'}(c)\\
:=&2\Im\big[\sum_{\left(j', j'_{1}, j'_{2}, j'_{3}\right) \in \mathcal{R}} \langle j'\rangle^{2} \big(\bar{u}^l_{j'} u^l_{j'_3} ( P_{\leq l_2}(u^h_{j'_1}\bar{u}^h_{j'_2}))-
\bar{u}^l_{j'} u^l_{j'_3} ( P_{\xi(t),\leq l_2}u^h_{j'_1})(P_{\xi(t),\leq l_2}\bar{u}^h_{j'_2})\big)\big]\\
=&2\Im\big[\sum_{\left(j', j'_{1}, j'_{2}, j'_{3}\right) \in \mathcal{R}} \langle j'\rangle^{2} \bar{u}^l_{j'} u^l_{j'_3}\big( ( P_{\leq l_2}(u^h_{j'_1}\bar{u}^h_{j'_2}))-
( P_{\xi(t),\leq l_2}u^h_{j'_1})(P_{\xi(t),\leq l_2}\bar{u}^h_{j'_2})\big)\big]\\
=&2\Im\big[\sum_{\left(j', j'_{1}, j'_{2}, j'_{3}\right) \in \mathcal{R}} \langle j'\rangle^{2} \bar{u}^l_{j'} u^l_{j'_3}\big( \int e^{ix\xi}\int ( \phi(\frac{\xi}{2^{l_2}})- \phi(\frac{\xi-\eta}{2^{l_2}})\phi(\frac{\eta}{2^{l_2}}) )\hat{u}^h_{j'_1}(\xi-\eta+\xi(t))\bar{\hat{u}}^h_{j'_2}(-\eta+\xi(t))d\eta d\xi\big)\big]\\
=:&\uppercase\expandafter{\romannumeral1}+\uppercase\expandafter{\romannumeral2}+\uppercase\expandafter{\romannumeral3},
\end{split}
\end{equation}
where
\begin{equation}\label{eq-ykl5.93'}
\begin{split}
\uppercase\expandafter{\romannumeral1}:=&2\Im\big[\sum_{\left(j', j'_{1}, j'_{2}, j'_{3}\right) \in \mathcal{R}} \langle j'\rangle^{2} \bar{u}^l_{j'} u^l_{j'_3}\big( \int e^{ix\xi}\int  \phi(\frac{\xi}{2^{l_2}})\psi(\frac{\xi-\eta}{2^{l_2}})\phi(\frac{\eta}{2^{l_2}}) \hat{u}^h_{j'_1}(\xi-\eta+\xi(t))\bar{\hat{u}}^h_{j'_2}(-\eta+\xi(t))d\eta d\xi\big)\big]\\
=&2\Im\big[\sum_{\left(j', j'_{1}, j'_{2}, j'_{3}\right) \in \mathcal{R}} \langle j'\rangle^{2} \bar{u}^l_{j'} u^l_{j'_3}P_{\leq l_2}\big( P_{\xi(t),l_2}u^h_{j'_1}P_{\xi(t),\leq l_2}\bar{u}^h_{j'_2}\big)\big].
\end{split}
\end{equation}
\begin{equation}\label{eq-ykl5.94'}
\begin{split}
\uppercase\expandafter{\romannumeral2}:=&2\Im\big[\sum_{\left(j', j'_{1}, j'_{2}, j'_{3}\right) \in \mathcal{R}} \langle j'\rangle^{2} \bar{u}^l_{j'} u^l_{j'_3}\big( \int e^{ix\xi}\int  \phi(\frac{\xi}{2^{l_2}})\phi(\frac{\xi-\eta}{2^{l_2}})\psi(\frac{\eta}{2^{l_2}}) \hat{u}^h_{j'_1}(\xi-\eta+\xi(t))\bar{\hat{u}}^h_{j'_2}(-\eta+\xi(t))d\eta d\xi\big)\big]\\
=&2\Im\big[\sum_{\left(j', j'_{1}, j'_{2}, j'_{3}\right) \in \mathcal{R}} \langle j'\rangle^{2} \bar{u}^l_{j'} u^l_{j'_3}P_{\leq l_2}\big( P_{\xi(t),\leq l_2}u^h_{j'_1}P_{\xi(t),l_2}\bar{u}^h_{j'_2}\big)\big].
\end{split}
\end{equation}
\begin{equation}\label{eq-ykl5.95'}
\begin{split}
\uppercase\expandafter{\romannumeral3}:=&2\Im\big[\sum_{\left(j', j'_{1}, j'_{2}, j'_{3}\right) \in \mathcal{R}} \langle j'\rangle^{2} \bar{u}^l_{j'} u^l_{j'_3}\big( \int e^{ix\xi}\int  \phi(\frac{\xi}{2^{l_2}})(1-\phi(\frac{\xi-\eta}{2^{l_2}}))(1-\phi(\frac{\eta}{2^{l_2}}))\\
&\quad\times \hat{u}^h_{j'_1}(\xi-\eta+\xi(t))\bar{\hat{u}}^h_{j'_2}(-\eta+\xi(t))d\eta d\xi\big)\big]\\
=&2\Im\big[\sum_{\left(j', j'_{1}, j'_{2}, j'_{3}\right) \in \mathcal{R}} \langle j'\rangle^{2} \bar{u}^l_{j'} u^l_{j'_3}P_{\leq l_2}\big( P_{\xi(t),\geq l_2}u^h_{j'_1}P_{\xi(t),\geq  l_2}\bar{u}^h_{j'_2}\big)\big].
\end{split}
\end{equation}
The estimate of $\uppercase\expandafter{\romannumeral1}$ and $\uppercase\expandafter{\romannumeral2}$ are similar, we just give the estimate of  $\uppercase\expandafter{\romannumeral1}$ below.

Recall $u_{j'} =u^{h}_{j'} + u^l_{j'}$ and $u^l_{j'}=P_{\xi(t),\leq l_2-5}u_{j'}$.
$supp \hat{u}_{j'_1}\subset \{2^{l_2-1}\leq|\xi-\eta-\xi(t)|\leq 2^{l_2+1}\}\subset \{2^{l_2-2}\leq|\xi-\eta-\xi(G_{\beta}^{l_2})|\leq 2^{l_2+2}\}$ and $supp (\bar{u}_{j'}u_{j'_3})^{\wedge}\subset \{|\xi|\leq 2^{l_2-4}\}$ infer that $supp \hat{\bar{u}}_{j'_2}\subset \{|\eta-\xi(G_{\beta}^{l_2})|\geq 2^{l_2-3}\}$, by the bilinear estimate Proposition \ref{pr-y5.5},  we have
 \begin{equation}\label{eq-ykl5.92bc}
\begin{split}
& 2^{l_2-i}\| \uppercase\expandafter{\romannumeral1}\|_{L^1_{t,x}(G_{\beta}^{l_2}\times \mathbb{R}^2)}\\
\lesssim& 2^{l_2-i}\Big\|\|\vec{u}^l\|_{h^1}\|\vec{u}_{\xi(G_{\beta}^{l_2}),\geq l_2-2}^h\|_{h^1}\Big\|_{L^{2}_{t}L^2_x(G_{\beta}^{l_2}\times \mathbb{R}^2)} \Big\|\|\vec{u}^l\|_{h^1}\|\vec{u}_{\xi(G_{\beta}^{l_2}),\geq l_2-3}^h\|_{h^1}\Big\|_{L^{2}_{t}L^2_x(G_{\beta}^{l_2}\times \mathbb{R}^2)}\\
\lesssim&2^{l_2-i}\|\vec{u}\|^4_{\tilde{X}_{i}(G_{\alpha}^{i})}.
\end{split}
\end{equation}

Finally we turn to  $\uppercase\expandafter{\romannumeral3}$. 
We can use the bilinear estimate to deal with $\uppercase\expandafter{\romannumeral3}$,
\begin{equation}\label{eq-ykl5.92c}
\begin{split}
& 2^{l_2-i}\|\uppercase\expandafter{\romannumeral3}\|_{L^1_{t,x}(G_{\beta}^{l_2}\times \mathbb{R}^2)}\\
\lesssim& 2^{l_2-i}\Big\|\|\vec{u}^l\|_{h^1}\|\vec{u}_{\xi(G_{\beta}^{l_2}),\geq l_2-2}^h\|_{h^1}\Big\|_{L^{2}_{t}L^2_x(G_{\beta}^{l_2}\times \mathbb{R}^2)} \Big\|\|\vec{u}^l\|_{h^1}\|\vec{u}_{\xi(G_{\beta}^{l_2}),\geq l_2-2}^h\|_{h^1}\Big\|_{L^{2}_{t}L^2_x(G_{\beta}^{l_2}\times \mathbb{R}^2)}\\
\lesssim&2^{l_2-i}\|\vec{u}\|^4_{\tilde{X}_{i}(G_{\alpha}^{i})}.
\end{split}
\end{equation}
Now we have finished the estimates of the distinctive terms.
Therefore by going on repeating the arguments in section 6 in \cite{YZ}, we have the following three bilinear Strichartz estimates.
\begin{theorem}[First bilinear Strichartz estimate]\label{th-y5.16}
For $a\in\{-1,1\}$, suppose that $\vec{v}_0=\{v_{0,j}\}_{j\in \mathbb{Z}^2}\in L_x^2h^{a} (\mathbb{R}^2\times \mathbb{Z}^2 )$ and $\hat{v}_{0,j}$ is supported on
$\{\xi:2^{i-5}\leq |\xi - \xi(G_{\alpha}^{i})|\leq 2^{i+5} \}$ for every $j\in \mathbb{Z}^2$.  Also suppose $J_l\subset G^i_{\alpha}$ is a small interval and $|\xi(t)-\xi(G^i_{\alpha})|\leq 2^{i-10}$ for all $t\in G^i_{\alpha}$. Then for any $0\leq l_2 \leq i-10$,
\begin{equation}\label{eq-y5.51}
\begin{split}
  &\left\| \|e^{it\triangle}\vec{v}_0\|_{h^{a}} \|P_{\xi(t),\leq l_2} \vec{u}\|_{h^{1}} \right\|^2_{L_t^2L_x^2(J_l\times\mathbb{R}^2)} \lesssim 2^{l_2-i}\|\vec{v}_0\|^2_{L_x^2h^{a}(\mathbb{R}^2\times \mathbb{Z}^2)}  \\
  +& 2^{-i}\|\vec{v}_0\|^2_{L_x^2h^{a}(\mathbb{R}^2\times \mathbb{Z}^2)} \big(\int_{J_l} |\xi'(t)|\sum_{l_1\leq l_2}2^{\frac{l_1-l_2}{2}} \|P_{\xi(t),l_2-3 \leq \cdot\leq l_2+3} \vec{u}\|_{L_x^2h^{1}(\mathbb{R}^2\times \mathbb{Z}^2)} \|P_{\xi(t), l_1} \vec{u}\|_{L_x^2h^{1}(\mathbb{R}^2\times \mathbb{Z}^2)} dt\big).
\end{split}
\end{equation}
The same estimate also holds when $P_{\xi(t),\leq l_2}$ is replaced by $P_{\xi(t), l_2}$.
\end{theorem}
\begin{theorem}[Second bilinear Strichartz estimate]\label{th-y5.17}
For $a\in\{-1,1\}$, suppose that $\vec{v}_0=\{v_{0,j}\}_{j\in \mathbb{Z}^2}\in L_x^2h^{a} (\mathbb{R}^2\times \mathbb{Z}^2 )$, $\hat{v}_{0,j}$ supported on
$\{\xi:2^{i-5}\leq |\xi - \xi(G_{\alpha}^{i})|\leq 2^{i+5} \}$ for every $j\in \mathbb{Z}^2$.  Then for any $0\leq l_2 \leq i-10$, $G_{\beta}^{l_2}\subset G^i_{\alpha}$,
\begin{equation}\label{eq-y5.67}
\left\| \|e^{it\triangle}\vec{v}_0\|_{h^{a}} \cdot\|P_{\xi(t),\leq l_2} \vec{u}\|_{h^{1}} \right\|^2_{L_t^2L_x^2(G_{\beta}^{l_2} \times\mathbb{R}^2)} \lesssim \|\vec{v}_0\|^2_{L_x^2h^{a}(\mathbb{R}^2\times \mathbb{Z}^2)}(1+\|\vec{u}\|^4_{\tilde{X}_{i}(G_{\alpha}^{i}\times\mathbb{R}^2)}).
\end{equation}
\end{theorem}
\begin{theorem}[Third bilinear Strichartz estimate]\label{th-y5.19}For $a\in\{-1,1\}$, suppose that $\vec{v}_0=\{v_{0,j}\}_{j\in \mathbb{Z}^2}\in L_x^2h^{a} (\mathbb{R}^2\times \mathbb{Z}^2 )$, $\hat{v}_{0,j}$ supported on
$\{\xi:2^{i-5}\leq |\xi - \xi(G_{\alpha}^{i})|\leq 2^{i+5} \}$ for every $j\in \mathbb{Z}^2$. Then we have
\begin{equation}\label{eq-y5.97}
\sum_{0\leq l_2\leq i-10}\big\|\|P_{\xi(t), \leq l_2} \vec{u}\|_{h^{1}} \cdot \|e^{it\triangle}\vec{v}_0\|_{h^{a}}\big\|^2_{L^2_{t,x}(G^{i}_{\alpha}\times\mathbb{R}^2)}\lesssim \|\vec{v}_0\|^2_{L^2_xh^{a}}(1+\|\vec{u}\|^6_{\tilde{X}_i(G^{i}_{\alpha})}).
\end{equation}
\end{theorem}

Using the above three bilinear Strichartz estimates and repeating the argument of Section 5 in \cite{YZ}, we have
\begin{theorem}[Long time Strichartz estimate]\label{th-y6.2}
Suppose $\vec{u}(t)$ is an almost periodic solution to
\begin{equation}\label{eq-y1}
\begin{cases}
i\partial_t u_j + \Delta_{\mathbb{R}^2} u_j = \sum\limits_{\mathcal{R}(j)} u_{j_1} \bar{u}_{j_2} u_{j_3},\\
u_j(0) = u_{0,j},
\end{cases}
\end{equation}for $\vec{u}_{0}=\{u_{0,j}\}_{j\in\mathbb{Z}^2}\in L^2_xh^1(\mathbb{R}^2\times\mathbb{Z}^2).$
Then there exists a constant $C>0$ (only depending on $\vec{u}$), such that for any $M=2^{k_0}$, $\epsilon_1,\epsilon_2,\epsilon_3$ satisfying above conditions, $\|\vec{u}\|^4_{L^4_{t,x}h^{1}([0,T])}=M$ and $\int_0^T N(t)^3 dt=\epsilon_3M$,
$$\|\vec{u}\|_{\tilde{X}_{k_0}([0,T])}\leq C.$$
\end{theorem}
\begin{remark}
Throughout this section the implicit constant depends only on $\vec{u}$, and not on $M$, or $\epsilon_1,\epsilon_2,\epsilon_3$.
\end{remark}

\subsection{Frequency localized interaction Morawetz estimate}

In this section we will prove frequency localized interaction Morawetz estimate, which are used to complete the proof of Theorem \ref{mainthm}.

Suppose $[0,T]$ is an interval such that, for some integer $k_0$,
$\|\vec{u}\|^4_{L^4_{t,x}h^1([0,T])}=2^{k_0}$. Rescale with $\lambda=\frac{\epsilon_3 2^{k_0}}{K}$, then by Theorem \ref{th-y6.2}, we have
\begin{equation}\label{eq-y6.1}
\|\vec{u}_{\lambda}\|_{\tilde{X}_{k_0}([0,\frac{T}{\lambda^2}]\times \mathbb{R}^2)}\lesssim 1.
\end{equation}
Let $\vec{w}=P_{\leq k_0}\vec{u}$, then $\vec{w}=\{w_j\}_{j\in\mathbb{Z}^2}$ satisfies the following infinite dimensional vector-valued equation
\begin{equation*}
  i\partial_t w_{j}+\triangle w_{j}=\sum_{\mathcal{R}(j)}w_{j_1}\bar{w}_{j_2}w_{j_3}+N_{j},
\end{equation*}
where $N_{j}=P_{\leq k_0}\big(\sum_{\mathcal{R}(j)}u_{j_1}\bar{u}_{j_2}u_{j_3}\big)- \sum_{\mathcal{R}(j)}w_{j_1}\bar{w}_{j_2}w_{j_3}$ and we denote $\vec{N}=\{N_{j}\}_{j\in \mathbb{Z}^2}$.

For $a\in\{0,2\}$, let
\begin{equation}\label{eq-y5.52'''}
\begin{split}
M(t)=\sum_{j,j'\in \mathbb{Z}^2} \langle j \rangle^a \langle j' \rangle^a \big(\int_{\mathbb{R}^2} \int_{\mathbb{R}^2} |w_{j'}(t,y)|^2 \frac{(x-y)}{|x-y|}\cdot Im[\bar{w}_j\nabla w_j](t,x)dxdy \big).
\end{split}
\end{equation}
Following the calculation of \cite{FL} and \cite{D}, we can show
\begin{equation}\label{eq-y9.2}
 \|\sum_{j\in\mathbb{Z}^2}\langle j \rangle^a |\nabla|^{1/2}|w_j(t,x)|^2\|^2_{L^2_{t,x}([0,\frac{T}{\lambda^2}]\times \mathbb{R}^2)} \lesssim \sup_{[0,\frac{T}{\lambda^2}]} |M(t)| + E,
\end{equation}
where $E$ is a Galilean invariant quantity. After Galilean transformation,
\begin{align}\label{eq-y9.3}
  E= & 2\big| \sum_{j,j'\in \mathbb{Z}} \langle j \rangle^a \langle j' \rangle^a \int^{\frac{T}{\lambda^2}}_{0} \int_{\mathbb{R}^2} \int_{\mathbb{R}^2} Im[\bar{w}_{j} (\nabla- i\xi(t))w_{j}](t,x) \frac{(x-y)}{|x-y|}\cdot Im[\bar{w}_{j'}N_{j'}](t,y)dxdydt \big| \\ \label{eq-y5.54''}
  + & \big| \sum_{j,j'\in \mathbb{Z}} \langle j \rangle^a \langle j' \rangle^a  \int^{\frac{T}{\lambda^2}}_{0} \int_{\mathbb{R}^2} \int_{\mathbb{R}^2} |w_j(t,y)|^2 \frac{(x-y)}{|x-y|}\cdot Im[\bar{N}_{j'} (\nabla- i\xi(t))w_{j'}](t,x)dxdydt \big|\\ \label{eq-y5.55''}
  + & \big| \sum_{j,j'\in \mathbb{Z}} \langle j \rangle^a \langle j' \rangle^a \int^{\frac{T}{\lambda^2}}_{0} \int_{\mathbb{R}^2} \int_{\mathbb{R}^2} |w_j(t,y)|^2 \frac{(x-y)}{|x-y|}\cdot Im[\bar{w}_{j'} (\nabla- i\xi(t))N_{j'}](t,x)dxdydt \big| .
\end{align}
Repeating the argument in Section 7 of \cite{YZ} and using the three cancellations we observed  in last section, we can estimate
$$\sup_{[0,\frac{T}{\lambda^2}]} |M(t)| + E\lesssim o(K),$$
where $\int_0^T N(t)^3 dt=K$ and $o(K)$ is a quantity such that $\frac{o(K)}{K}\rightarrow0 $ as $K\rightarrow \infty$.
Thus as in \cite{YZ}, we have

\begin{theorem}[Frequency localized interaction Morawetz estimate]\label{th-y9.1}
   Suppose $\vec{u}(t,x)$ is a minimal mass blowup solution to \eqref{eq-y1} on $[0,T]$ with
$\int_0^T N(t)^3 dt=K$.  Then for $a\in\{0,2\}$, we have
\begin{equation}\label{eq-y9.1}
 \|\sum_{j\in\mathbb{Z}^2}\langle j \rangle^a |\nabla|^{1/2}|P_{\leq \frac{10K}{\epsilon_1}}u_j(t,x)|^2\|^2_{L^2_{t,x}([0,T]\times \mathbb{R}^2)} \lesssim o(K),
\end{equation}
where $o(K)$ is a quantity such that $\frac{o(K)}{K}\rightarrow0 $ as $K\rightarrow \infty$.
\end{theorem}

\section{The Rigidity Theorem}
In this section, we prove the rigidity theorem, which is the last step to prove the scattering for the initial value problem \eqref{maineq}. The rigidity theorem and its proof are as below. We will discuss two scenarios respectively.
\begin{theorem}\label{th-y9.2}
   There does not exist a minimal mass blowup solution to \eqref{eq-y1}.
\end{theorem}
\begin{proof}It  suffices to exclude two scenarios separately.\\
\textbf{Case 1.: Rapid frequency cascade: $\int_0^{\infty} N(t)^3 dt <\infty$.}

In this case, we can repeat the process as section 5 of \cite{D} and follow the arguments in Section 6 in that paper to obtain an additional regularity of a minimal mass blowup solution to \eqref{eq-y1}, that is, $\|\vec{u}(t,x)\|_{L^{\infty}_t \dot{H}^{3}_xh^1([0,\infty)\times \mathbb{R}^2)} \lesssim_{m_0} (\int^{\infty}_0 N(t)^3 dt)^3$, which together with the definition of almost periodic solution yields
$$\|e^{-ix\cdot\xi(t)}\vec{u}\|_{\dot{H}^{1}_xh^1}\lesssim N(t)C(\eta(t))+\eta(t)^{1/2},\quad \eta(t)\rightarrow 0.$$
Since $\lim_{t\rightarrow\infty}N(t)=0$, this implies $\lim_{t\rightarrow\infty}\|e^{-ix\cdot\xi(t)}\vec{u}\|_{\dot{H}^{1}_xh^1}=0$. So for any $\epsilon>0$, there exists $t_0>0$ (by Galilean transformation we may take $t_0=0$), such that $\|e^{-ix\cdot\xi(t_0)}\vec{u}(t_0)\|_{\dot{H}^{1}_xh^1}<\epsilon$. \\
Notice that
$$E(\vec{u}(t))=\frac{1}{2}\int_{\mathbb{R}^{2}}\sum\limits_{j\in\mathbb{Z}^2}|\nabla u_{j}(t,x)|^{2}\mathrm{d}x+\frac{1}{4}\int_{\mathbb{R}^{2}}\sum\limits_{\substack{j_0,j_1,j_2,j_3\in \mathbb{Z}^2,\\ j_1-j_2+j_3 = j_0,\\ |j_1|^2-|j_2|^2 +|j_3|^2 = |j_0|^2.}} \bar{u}_{j_0}u_{j_1} \bar{u}_{j_2} u_{j_3}\mathrm{d}x=E(\vec{u}(0)),$$
by Minkowski inequality and sharp Gagliardo-Nirenberg inequality, we can calculate
\begin{equation*}
\begin{split}
\int_{\mathbb{R}^{2}}\sum\limits_{\substack{j_0,j_1,j_2,j_3\in \mathbb{Z}^2,\\ j_1-j_2+j_3 = j_0,\\ |j_1|^2-|j_2|^2 +|j_3|^2 = |j_0|^2.}} \bar{u}_{j_0}u_{j_1} \bar{u}_{j_2} u_{j_3}\mathrm{d}x \lesssim& \int_{\mathbb{R}^{2}}\big( \sum_{j}\langle j\rangle^2|u_j|^2\big)^2\mathrm{d}x\lesssim  \big(\sum_{j} \langle j\rangle^2 \|u_j\|^2_{L^4_{x}(\mathbb{R}^{2})}\big)^2\\
\lesssim & \big(\sum_{j} \langle j\rangle^2 \|u_j\|_{L^2_{x}(\mathbb{R}^{2})} \|\nabla u_j\|_{L^2_{x}(\mathbb{R}^{2})}\big)^2\\
\lesssim & (\sum_{j} \langle j\rangle^2 \|u_j\|^2_{L^2_{x}(\mathbb{R}^{2})}) (\sum_{j}\langle j\rangle^2\|\nabla u_j\|^2_{L^2_{x}(\mathbb{R}^{2})}).
\end{split}
\end{equation*}Therefore, $E(\vec{u}(t))=E(\vec{u}(0))\lesssim \|\vec{u}(0)\|^2_{\dot{H}^{1}_xh^1
}< \ep^2$.\\
However, by H\"{o}lder inequality,
\begin{align*}
  \sum_{j\in\mathbb{Z}^2}\langle j\rangle^2\int|u_j(0,x)|^2dx \leq& \sum_{j\in\mathbb{Z}^2}\langle j\rangle^2 \int_{|x-x(0)|\leq \frac{C\big(\frac{\sum_j\langle j\rangle^2\|u_j(0)\|^2_{L^2}}{1000}\big)}{N(0)}}|u_j(0,x)|^2dx +\frac{\sum_j\langle j\rangle^2\|u_j(0)\|^2_{L^2}}{1000}\\
   \leq& C[\int_{\mathbb{R}^{2}}\big( \sum_{j}\langle j\rangle^2|u_j(0,x)|^2\big)^2dx]^{1/2} \frac{C\big(\frac{\sum_j\langle j\rangle^2\|u_j(0)\|^2_{L^2}}{1000}\big)}{N(0)} +\frac{\sum_j\langle j\rangle^2\|u_j(0)\|^2_{L^2}}{1000} \\
   \leq& CE(\vec{u}(0))^{1/2}\frac{C\big(\frac{\sum_j\langle j\rangle^2\|u_j(0)\|^2_{L^2}}{1000}\big)}{N(0)} +\frac{\sum_j\langle j\rangle^2\|u_j(0)\|^2_{L^2}}{1000}.
\end{align*}
Choose $\ep$ sufficiently small such that
$$C \ep \frac{C\big(\frac{\sum_j\langle j\rangle^2\|u_j(0)\|^2_{L^2}}{1000}\big)}{N(0)}<\frac{\sum_j\langle j\rangle^2\|u_j(0)\|^2_{L^2}}{100}.$$
This implies $\sum_{j\in\mathbb{Z}^2}\langle j\rangle^2\int|u_j(0,x)|^2dx<\frac{\sum_{j\in\mathbb{Z}^2}\langle j\rangle^2\int|u_j(0,x)|^2dx}{100}$, which can't happen unless $\sum_{j\in\mathbb{Z}^2}\langle j\rangle^2\int|u_j(0,x)|^2dx=0$, so $\int|u_j(0,x)|^2dx=0$ for every $j\in\mathbb{Z}^2$, this infers $\|\vec{u}(0)\|_{L^2_xh^1}=0$. This excludes rapid frequency cascade scenario.\vspace{3mm}

\noindent\textbf{Case 2. Quasi-soliton: $\int_0^{\infty} N(t)^3 dt =\infty$.}

In this case, we denote $Iu_j=P_{\leq10\epsilon_1^{-1}K}u_j$. By frequency localized interaction Morawetz estimate,
$$\|\sum_{j\in\mathbb{Z}}|\nabla|^{1/2}|Iu_j(t,x)|^2\|^2_{L^2_{t,x}([0,T]\times \mathbb{R}^2)} \lesssim o(K),$$
where recalling $\int^T_0 N(t)^3 dt=K$.
By H\"{o}lder inequality and the Sobolev embedding, we have
\begin{align*}
&\sum_{j\in\mathbb{Z}}\int_{|x-x(t)|\leq \frac{C\big(\frac{\sum_j\|u_j\|^2_{L^2}}{1000}\big)}{N(t)}}|Iu_j(t,x)|^2dx \\
\lesssim& \big(\frac{C\big(\frac{\sum_j\|u_j\|^2_{L^2}}{1000}\big)}{N(t)}\big)^{3/2} \|\sum_j|Iu_j(t)|^2\|_{L^4_x}\\
\lesssim&\big(\frac{C\big(\frac{\sum_j\|u_j\|^2_{L^2}}{1000}\big)}{N(t)}\big)^{3/2} \|\sum_{j\in\mathbb{Z}}|\nabla|^{1/2}|Iu_j(t,x)|^2\|_{L^2_{x}}.
\end{align*}
Now for $K>C\big(\frac{\sum_j\|u_j\|^2_{L^2}}{1000}\big)$, by Proposition \ref{pr-y4.8}, we have
$$\frac{\sum_j\|u_j\|^2_{L^2}}{2}<\sum_{j\in\mathbb{Z}^2}\int_{|x-x(t)|\leq \frac{C\big(\frac{\sum_j\|u_j\|^2_{L^2}}{1000}\big)}{N(t)}}|Iu_j(t,x)|^2dx.$$
Therefore,
\begin{align*}
&\big(\sum_j\|u_j\|^2_{L^2}\big)^2 K\sim \big(\sum_j\|u_j\|^2_{L^2}\big)^2\int_0^T N(t)^3dt\\
\lesssim &\int_0^T N(t)^3 \left(\sum_{j\in\mathbb{Z}^2}\int_{|x-x(t)|\leq \frac{C\big(\frac{\sum_j\|u_j\|^2_{L^2}}{1000}\big)}{N(t)}}|Iu_j(t,x)|^2dx\right)^2 dt\\ \lesssim&\|\sum_{j\in\mathbb{Z}}|\nabla|^{1/2}|Iu_j(t,x)|^2\|^2_{L^2_{t,x}([0,T]\times \mathbb{R}^2)}\lesssim o(K).
\end{align*}
Combined with the mass conservation, this gives a contradiction for $K$ sufficiently large. Therefore, the proof of Theorem
\ref{th-y9.2} is complete.
\end{proof}
The proof of Theorem \ref{mainthm} is also complete in view of the standard concentration compactness method.
\section{Appendix}
\subsection{Further remarks}
In this subsection, we make a few remarks on the large data scattering for the defocusing critical NLS. We also include some related problems for interested readers.

We discuss and summarize a specific type of problems by making some restrictions. We consider a series of problems, i.e. large data scattering for the defocusing critical NLS with integer index nonlinearity on low dimensional (when $m+n \leq 4$) waveguides as follows
\begin{align}
(i\partial_t+ \Delta_{\mathbb{R}^{m} \times \mathbb{T}^{n}}) u &= F(u) = |u|^{p-1} u, \\
u(0,x) &= u_{0} \in H^{1}(\mathbb{R}^{m} \times \mathbb{T}^{n}),
\end{align}
where $\Delta_{\mathbb{R}^m\times \mathbb{T}^n}$ is the Laplace-Beltrami operator on $\mathbb{R}^m\times \mathbb{T}^n$ and  $u:\mathbb{R}\times \mathbb{R}^m\times \mathbb{T}^n \rightarrow \mathbb{C}$ is a complex-valued function.

Based on existing results and theories, it is expected that only when $\frac{4}{m}\leq p \leq \frac{4}{m+n-2}$, scattering behavior is expected to hold. (see \cite{IPRT3,HP,Z1} for explanations).

Noticing the range $\frac{4}{m}\leq p \leq \frac{4}{m+n-2}$, we have $n=0,1,2$. In fact, there are totally $11$ models. When $n=0$ (pure Euclidean case),
\begin{enumerate}
\item $(i\partial_t+ \Delta_{\mathbb{R} }) u = F(u) = |u|^{4} u$ \quad
\item $(i\partial_t+ \Delta_{\mathbb{R}^{2} }) u = F(u) = |u|^{2} u$ \quad
\item $(i\partial_t+ \Delta_{\mathbb{R}^{3} }) u = F(u) = |u|^{4} u$ \quad
\item $(i\partial_t+ \Delta_{\mathbb{R}^{4} }) u = F(u) = |u|^{2} u$ \quad
\item $(i\partial_t+ \Delta_{\mathbb{R}^{4} }) u = F(u) = |u|u$ \quad

\end{enumerate}
The results for the Euclidean case are well-known.

When $n=1$, there are 4 cases
\begin{enumerate}
\item $(i\partial_t+ \Delta_{\mathbb{R} \times \mathbb{T}}) u = F(u) = |u|^{4} u$ \quad solved by  Cheng-Guo-Zhao \cite{CGZ} (mass critical and energy subcritical)
\item $(i\partial_t+ \Delta_{\mathbb{R}^{2} \times \mathbb{T}}) u = F(u) = |u|^{4} u$ \quad solved by Zhao \cite{Z2} (mass supercritical and energy critical)
\item $(i\partial_t+ \Delta_{\mathbb{R}^{3} \times \mathbb{T}}) u = F(u) = |u|^{2} u$  \quad solved by Zhao \cite{Z2} (mass supercritical and energy critical)
\item $(i\partial_t+ \Delta_{\mathbb{R}^{2} \times \mathbb{T}}) u = F(u) = |u|^{2} u$  \quad solved by Cheng-Guo-Yang-Zhao \cite{CHENG} (mass critical and energy subcritical)
\end{enumerate}

\noindent When $n=2$, there are 2 cases
\begin{enumerate}
\item $(i\partial_t+ \Delta_{\mathbb{R}^{2} \times \mathbb{T}^{2}}) u = F(u) = |u|^{2} u$ \quad solved by Zhao \cite{Z1} (mass critical and energy critical) (assuming scattering for the 2D cubic resonant system)
\item $(i\partial_t+ \Delta_{\mathbb{R} \times \mathbb{T}^{2}}) u = F(u) = |u|^{4} u$ \quad  solved by Hani-Pausader \cite{HP} (mass critical and energy critical)
\end{enumerate}

The scattering result for this 2D cubic resonant system was the last missing brick, which is proved in this current paper. This category of problems are now all solved.\vspace{3mm}

At last, we conclude this paper with a few more remarks.\vspace{3mm}

\begin{remark}
The reason that most of the results on critical NLS on waveguides concerns low dimensional space and integer nonlinear exponent is the technical restriction of the function spaces ($U^p$, $V^p$ spaces). It is surely interesting to consider the high dimensional case and the fractional nonlinearity case (they often coincide together).  As concluded in \cite{IPRT3}, in general, the difficulty of the critical NLS problem on $\mathbb{R}^m \times \mathbb{T}^n$ increases if the dimension $m+n$ is increased or if the number $m$ of copies of $\mathbb{R}$ is decreased.
\end{remark}
\begin{remark}
Large data scattering for the focusing NLS on waveguides are comparably less understood than the defocusing case. Threshold assumptions are necessary and new ingredients are needed to handle this type of problems. See \cite{YYZ} for a global well-posedness result and \cite{Foc2,Foc1} for two very recent scattering result. Moreover, see \cite{D5,KM,KV2} for the Euclidean result. It is interesting to consider the focusing analogue of the results in this current paper.
\end{remark}
\begin{remark}
One may consider other related problems such as `other dispersive equations on waveguides' and `NLS on other manifolds/product spaces'. See \cite{Hari,SYYZ,YYZ2} for examples (Klein-Gordon equations on waveguides, Fractional NLS on waveguides and Fourth order NLS on waveguides respectively).
\end{remark}
\subsection{Weaker estimates}
In this subsection, we discuss the discrete estimates for the resonant nonlinearity from another aspect. We  recall the cubic resonances $\vec{F}(u)=\{ \sum_{(j_1,j_2,j_3)\in R(j)} u_{j_{1}}\bar{u}_{j_{2}}u_{j_{3}} \}_{j}$. We expect to show: for some positive $\delta_1,\delta_2$, $0<\beta<1$,
\begin{equation}
    \|\vec{F}(u)\|_{l^2} \lesssim \|\vec{u}\|^{\delta_1}_{l^2}\|\vec{u}\|^{3-\delta_1}_{h^1},
\end{equation}
\begin{equation}
           \|\vec{F}(u)\|_{l^2} \lesssim \|\vec{u}\|_{l^2}\|\vec{u}\|^{2}_{h^{\beta}}.
\end{equation}
and
\begin{equation}
           \|\vec{F}(u)\|_{h^1} \lesssim \|\vec{u}\|^{\delta_2}_{l^2}\|\vec{u}\|^{3-\delta_2}_{h^1}.
\end{equation}
The above estimates have their own interests. Some number theory will be involved. Rather than the scattering, one may consider other problems related to the cubic resonant system. We will discuss the proofs for the above estimates.

\begin{lemma}There holds that
\begin{equation}\label{eq-yle3.3}
\sup\limits_{j \in \mathbb{Z}^2}\{\langle j \rangle^2 \sum_{R(j),p_3 \textmd{ largest }}\langle p_1 \rangle^{-2\beta}\langle p_2 \rangle^{-2\beta}\langle p_3 \rangle^{-2}\} \lesssim 1.
\end{equation}
and
\begin{equation}\label{eq-yle3.3'}
\sup\limits_{j \in \mathbb{Z}^2}\{\sum_{R(j),p_3 \textmd{ largest }}\langle p_1 \rangle^{-2\beta}\langle p_2 \rangle^{-2\beta}\} \lesssim 1.
\end{equation}
\end{lemma}
\noindent \emph{Proof:} We just prove \eqref{eq-yle3.3} below because the other is similar, without loss of generality, we may assume that
\[|p_1|\leq |p_3|, \quad max(|j|,|p_2|)\sim |p_3|.
\]
\noindent Also we can see that $p_1$ is on a specific circle $\mathcal{C}$,
\[|p_1-\frac{p_2-j}{2}|^2=(\frac{p_2-j}{2})^2.
\]
\begin{align*}
S_1 &=\sum_{(p_1,p_2,p_3)\in R(j);|p_1|\leq|p_3|;|p_2|\leq|p_1|} \langle p_1 \rangle^{-2\beta} \langle p_2 \rangle^{-2\beta} \frac{\langle j \rangle^{2}}{\langle p_3 \rangle^{2}}\\
&\lesssim \sum_{(p_1,p_2,j+p_2-p_1)\in R(j);|p_2|\leq|p_1|} \langle p_1 \rangle^{-2\beta} \langle p_2 \rangle^{-2\beta} [\frac{\langle j \rangle}{\langle max(|j|,|p_2|) \rangle}]^2\\
&\lesssim \sum_{p_2} \langle p_2 \rangle^{-2\beta} \sum_{p_1} \langle p_1 \rangle^{-2\beta}\\
&\lesssim \sum_{p_2} \langle p_2 \rangle^{-2\beta} \langle |p_2| \rangle^{1-2\beta} \lesssim 1.
\end{align*}

\noindent The sum when $|p_1|\leq|p_2|$ is bounded similarly, using the following lemma to bound the sum over $p_2$ instead of the bound over $p_1$.

\begin{lemma}For any $P\in \mathbb{R}^2$, $R>0$ and $A>1$ there hold that:
\[\sum_{|p|\geq A,p\in \mathbb{Z}^2\cap C(P,R)}\frac{1}{\langle p \rangle^{2\beta}} \lesssim A^{1-2\beta}
\]
where $C(P,R)$ denotes the circle of radius $R$ centered at $P$.
\end{lemma}
We note that we need to assume $4\beta-1>2$ because of the integrability, which means $\beta>\frac{3}{4}$. Then, according to the above lemma, we have,
\begin{equation}
           \|\vec{F}(u)\|_{h^1} \lesssim \|\vec{u}\|^{2}_{h^{\beta}}\|\vec{u}\|_{h^1}.
\end{equation}
By interpolation,
\begin{equation}
           \|\vec{F}(u)\|_{h^1} \lesssim \|\vec{u}\|^{\delta_2}_{l^2}\|\vec{u}\|^{3-\delta_2}_{h^1},
\end{equation}
where
\begin{equation}
    0<\delta_2=2-2\beta<\frac{1}{2}.
\end{equation}
Similarly, we can obtain,
\begin{equation}
           \|\vec{F}(u)\|_{l^2} \lesssim \|\vec{u}\|^{\delta_1}_{l^2}\|\vec{u}\|^{3-\delta_1}_{h^1},
\end{equation}
where
\begin{equation}
    0<\delta_1=3-2\beta<\frac{3}{2}.
\end{equation}
Moreover, we can obtain
\begin{equation}\label{eq-yle3.11}
           \|\vec{F}(u)\|_{l^2} \lesssim \|\vec{u}\|_{l^2}\|\vec{u}\|^{2}_{h^{\beta}}.\vspace{5mm}
\end{equation}

\noindent \textbf{Acknowledgments.}
We highly appreciate Prof. Chenjie Fan and Prof. Lifeng Zhao for helpful discussions and beneficial suggestions on this project. In particular, we are very grateful to Chenjie for discussing `the failure of $l^2$ estimate' for the cubic resonances in Section 3.

K. Yang was supported by a Doctoral Foundation of Chongqing Normal University (21XLB025) and a funding (6142A0521Q06, HX02021-36) from Laboratory  of  Computational  Physics, Institute of Applied Physics and Computational Mathematics in Beijing. Z. Zhao was supported by the NSF grant of China (No. 12101046) and the Beijing Institute of Technology Research Fund Program for Young Scholars.
\vspace{5mm}

%

%
\end{document}